\documentclass[a4paper]{article}

\usepackage{amsmath}
\usepackage{amssymb}
\usepackage{amsthm}
\usepackage{dsfont}
\usepackage{enumerate}
\usepackage{pifont}
\usepackage{theoremref}
\usepackage{tikz}
\usetikzlibrary{trees}
\usetikzlibrary{decorations.markings}

\newcommand{\mc}[1]{{\mathcal{#1}}}
\newcommand{\mf}[1]{{\mathfrak{#1}}}
\newcommand{\bb}[1]{{\mathbb{#1}}}

\DeclareMathOperator{\RE}{Re}
\DeclareMathOperator{\IM}{Im}
\renewcommand{\Re}{\RE}
\renewcommand{\Im}{\IM}

\DeclareMathOperator{\dom}{dom}
\DeclareMathOperator{\ran}{ran}
\newlength{\maxlabwidth}

\numberwithin{equation}{section}
\swapnumbers
\theoremstyle{plain}
	\newtheorem{lemma}{Lemma}[section]
	\newtheorem{proposition}[lemma]{Proposition}
	\newtheorem{theorem}[lemma]{Theorem}
	\newtheorem{corollary}[lemma]{Corollary}
	\newtheorem{ntheoreM}[lemma]{}
\theoremstyle{definition}
	\newtheorem{definitioN}[lemma]{Definition}
\theoremstyle{remark}
	\newtheorem{remarK}[lemma]{Remark}
	\newtheorem{examplE}[lemma]{Example}
	\newtheorem{nremarK}[lemma]{}
\newcommand{\thlab}[1]{\thlabel{#1}\label{#1.}}
\renewcommand{\qedsymbol}{\raisebox{-2pt}{\large\ding{113}}}% Box mit Schatten
\newcommand{\defendsymbol}{$\lozenge$}
\newcommand{\qedsymbolsave}{\qedsymbol}

\newenvironment{definition}{\begin{definitioN}}{
	\renewcommand{\qedsymbolsave}{\qedsymbol}\renewcommand{\qedsymbol}{\defendsymbol}
	\popQED{\qed}\renewcommand{\qedsymbol}{\qedsymbolsave}\end{definitioN}}
\newenvironment{remark}{\begin{remarK}}{
	\renewcommand{\qedsymbolsave}{\qedsymbol}\renewcommand{\qedsymbol}{\defendsymbol}
	\popQED{\qed}\renewcommand{\qedsymbol}{\qedsymbolsave}\end{remarK}}
\newenvironment{example}{\begin{examplE}}{
	\renewcommand{\qedsymbolsave}{\qedsymbol}\renewcommand{\qedsymbol}{\defendsymbol}
	\popQED{\qed}\renewcommand{\qedsymbol}{\qedsymbolsave}\end{examplE}}

\newcommand{\bibi}[5]{\bibitem[#5]{#1} \textsc{#2}:\ \textit{#3,}\ {#4.} }

\begin{document}
\begin{flushleft}
	{\Large\bf Spectral Theorem for definitizable normal linear operators on Krein spaces}
	\\[5mm]
	\textsc{Michael Kaltenb\"ack}
	\\[6mm]
	{\small
	\textbf{Abstract:In the present note a spectral theorem for normal 
		definitizable linear operators on Krein spaces is derived by developing
		a functional calculus $\phi \mapsto \phi(N)$ which is the 
		proper analogue of $\phi \mapsto \int \phi \, dE$ in the Hilbert space situation.}
	}
\end{flushleft}

\begin{flushleft}
\end{flushleft}
\begin{flushleft}
\end{flushleft}

\section{Introduction}

A bounded linear operator $N$ on a Krein space $(\mc K,[.,.])$ is called normal, if
$N$ commutes with its Krein space adjoint $N^*$, i.e.\ $NN^*=N^*N$. This
is equivalent to the fact that its real part $A:=\frac{N+N^*}{2}$ and its 
imaginary part $B:=\frac{N-N^*}{2i}$ commute. 
We call $N$ definitizable
whenever the selfadjoint operators $A$ and $B$ are both definitizable in classical sense,
i.e.\ there exist so-called definitizing polynomials $p(z)$ and $q(z)$ such that 
$[p(A)x,x] \geq 0$ and $[q(B)x,x] \geq 0$ for all $x\in \mc K$; see \cite{langer1982}. 

In the Hilbert space setting the spectral theorem for bounded linear, normal operators 
is a well-known functional analysis result. In fact, it is almost as 
as folklore as the older spectral theorem for bounded linear, selfadjoint operators.

In the Krein space world there exists no similar result for general selfadjoint operators.
But assuming in addition definitizability a spectral theorem could be shown by Heinz Langer;
cf.\ \cite{langer1982}. This theorem became an important starting point for various spectral 
results. The main difference to selfadjoint operators on Hilbert spaces is the appearance  
of (finitely many) critical points, where the spectral projections no longer behave like a measure. 

Only a rather small number of publications dealt with the situation of a normal 
(definitizable) operators in a Krein space.
The Pontryagin space case was studied up to a certain extent for example in 
\cite{xiaochao1985} and \cite{langerszafraniec2006}. Special normal operators on Krein spaces
were considered for example in \cite{azizovstrauss2003} and \cite{philippstrausstrunk2013}.
But until now no adequate version of a spectral theorem on normal definitizable operators in Krein spaces
has been found.

In the present paper we present a spectral theorem for bounded linear, normal, 
definitizable operators formulated in terms of a functional calculus generalizing 
the functional calculus $\phi\mapsto \int \phi \, dE$ in the Hilbert space case.
In order to achieve this goal, we use the methods developed in \cite{KaPr2014} 
for definitizable selfadjoint operators and extend them for two commuting
definitizable selfadjoint operators. 

Let us anticipate a little more explicitly what happens in this note. Denoting by 
$p(z)$ and $q(z)$ the definitizing real polynomials for $A$ and $B$, respectively, 
we build a Hilbert space $\mc V$ which is continuously and densely embedded in the given Krein space 
$\mc K$ such that $TT^* = p(A)+q(B)$, where $T: \mc V \to \mc K$ denotes that adjoint of the embedding
mapping. Then we use the $*$-homomorphism $\Theta: (TT^*)' \ (\subseteq B(\mc K)) \to (T^*T)' \ (\subseteq B(\mc V))$,
$C \mapsto (T\times T)^{-1}(C)$, studied in \cite{KaPr2014}, in order to drag our normal
operator $N\in (TT^*)'\subseteq B(\mc K)$ into $(T^*T)' \ (\subseteq B(\mc V)$. The resulting
normal operator $\Theta(N)$ acts in a Hilbert space, and therefore has a spectral measure
$E(\Delta)$, where $\Delta$ are Borel subsets of $\bb C$.

The proper family $\mc F_N$ of functions suitable for the aimed functional 
calculus are bounded and measurable functions on
\[
\big(\sigma(\Theta(N)) \cup (Z^{\bb R}_p + i Z^{\bb R}_q)\big)\dot\cup Z^i \ 
(\subseteq \bb C \dot\cup \bb C^2) \,.
\]
Here $Z^{\bb R}_p = p^{-1}\{0\}\cap \bb R$ and $Z^{\bb R}_q = q^{-1}\{0\}\cap \bb R$ denote the real 
zeros of $p(z)$ and $q(z)$, respectively, and $Z^i = (p^{-1}\{0\}\times q^{-1}\{0\}) \setminus (\bb R\times\bb R)$.
Moreover, the functions $\phi\in \mc F_N$ assume values in $\bb C$ on 
$\sigma(\Theta(N)) \setminus (Z^{\bb R}_p + i Z^{\bb R}_q)$, values in 
$\bb C^{\mf d_p(\Re z)\cdot \mf d_q(\Im z) + 2}$ at $z\in Z^{\bb R}_p + i Z^{\bb R}_q$ and 
values in $\bb C^{\mf d_p(\xi)\cdot \mf d_q(\eta)}$ at $z=(\xi,\eta)\in Z^i$.
Here $\mf d_p(w)$ ($\mf d_q(w)$) denotes $p$'s ($q$'s) degree of zero at $w$.
Finally, $\phi\in \mc F_N$ satisfies a growth regularity condition at all
points from $Z^{\bb R}_p + i Z^{\bb R}_q$ which are not isolated in
$\sigma(\Theta(N)) \cup (Z^{\bb R}_p + i Z^{\bb R}_q)$.

Any polynomial $s(z,w)\in \bb C[z,w]$ can be seen as a function $s_N\in \mc F_N$. 
The nice thing about these, somewhat tediously defined functions $\phi\in \mc F_N$ is that
\begin{equation}\label{decompact}
    \phi(z) = s_N(z) + 
	(p_N + q_N) (z)\cdot g(z), \ z \in \sigma(\Theta(N)) \,,
\end{equation}
where $s \in \bb C[z,w]$ is a suitable polynomial in two variables and 
$g: \sigma(\Theta(N))\setminus (Z^{\bb R}_p + i Z^{\bb R}_q) \to \bb C$
is bounded and measurable and $g: \sigma(\Theta(N))\cap (Z^{\bb R}_p + i Z^{\bb R}_q) \to \bb C^2$.

We then define $\phi(N):=s(A,B) + T \int_{\sigma(\Theta(N))}^{R_1,R_2} g \, dE T^*$, show that
this operator does not depend on the actual decomposition \eqref{decompact} and 
that $\phi \mapsto \phi(N)$ is indeed a $*$-homomorphism. Here 
$\int_{\sigma(\Theta(N))}^{R_1,R_2} g \, dE$ is the integral of $g$ with respect 
to the spectral measure $E$ taking into account the fact that $g$ has values in 
$\bb C^2$ on $\sigma(\Theta(N))\cap (Z^{\bb R}_p + i Z^{\bb R}_q)$.

If $\phi$ is stems from a characteristic function corresponding to a Borel subset $\Delta$ 
of $\bb C$ such that no point of $Z^{\bb R}_p + i Z^{\bb R}_q$ belongs to the boundary
of $\Delta$, then $\phi(N)$ is a selfadjoint projection on $\mc K$. In fact, it can be seen as
the corresponding special projection for $N$.

\section{Multiple embeddings}

For the present section we fix a Krein space $(\mc K,[.,.])$ and Hilbert spaces
$(\mc V,[.,.])$, $(\mc V_1,[.,.])$ and $(\mc V_2,[.,.])$. Moreover, let 
$T_1: \mc V_1 \to \mc K$, $T_2: \mc V_2 \to \mc K$
and $T: \mc V \to \mc K$ be bounded linear, injective mappings such that
\[
    TT^* = T_1T_1^* + T_2T_2^* 
\]
holds true. Since for $x\in \mc K$ we have
\begin{multline*}
	[T^*x,T^*x]_{\mc V} = [TT^*x,x] = \\ [T_1T_1^*x,x] + [T_2T_2^*x,x] = 
		  [T_1^*x,T_1^*x]_{\mc V_1} + [T_2^*x,T_2^*x]_{\mc V_2} \,,
\end{multline*}
one easily concludes that $T^* x \mapsto T_j^*x$ constitutes a well-defined, contractive linear mapping from 
$\ran T^*$ onto $\ran T_j^*$ for $j=1,2$. By $(\ran T^*)^\bot = \ker T=\{0\}$ and $(\ran T_j^*)^\bot = \ker T_j=\{0\}$ 
these ranges are dense in the Hilbert spaces $\mc V$ and $\mc V_j$. Hence, 
there is a unique bounded linear continuation of $T^* x \mapsto T_j^*x$ to $\mc V$, 
which has dense range in $\mc V_j$.

Denoting by $R_j$ for $j=1,2$ the adjoint mapping of this continuation we clearly have $T_j = T R_j$
and $\ker R_j = (\ran R_j^*)^\bot = \{0\}$. From $TT^* = T_1T_1^* + T_2T_2^*$ we conclude
\[
      T( \ I \ )T^* = TT^* = T R_1 R_1^*T^* + T R_2 R_2^*T^* = T( \ R_1 R_1^* + R_2 R_2^* \ )T^* \,.
\]
$\ker T=\{0\}$ and the density of $\ran T^*$ yields $R_1 R_1^* + R_2 R_2^* = I$.

If $T_1T_1^*$ and $T_2T_2^*$ commute, then by $TT^* = T_1T_1^* + T_2T_2^*$ also $T_jT_j^*$ and $TT^*$ commute.
Moreover, in this case
\[
    T ( \ T^* T R_j R_j^* \ )T^* = TT^* T_jT_j^* = T_jT_j^* TT^* = T ( \ R_j R_j^* T^* T \ )T^* \,.
\]
Employing again $T$'s injectivity and the density of $\ran T^*$ we see that $R_j R_j^*$ and $T^* T$ commute for 
$j=1,2$. From this we get
\[
    T_j^* T_j R_j^*R_j = R_j^* (T^* T R_j R_j^*) R_j = R_j^* (R_j R_j^* T^* T) R_j = R_j^*R_j T_j^* T_j \,.
\]
Thus, we showed

\begin{lemma}\thlab{existtreans}
    With the above notations and assumptions there exist injective contractions $R_1: \mc V_1 \to \mc V$
    and $R_2: \mc V_2 \to \mc V$ such that $T_1 = T R_1$, $T_2 = T R_2$ and
    $R_1 R_1^* + R_2 R_2^* = I$.
    
    If $T_1T_1^*$ and $T_2T_2^*$ commute, then the operators $R_j R_j^*$ and $T^* T$ on $\mc V$ commute
    as well as the operators $R_j^*R_j$ and $T_j^* T_j$ on $\mc V_j$ for $j=1,2$.
\end{lemma}

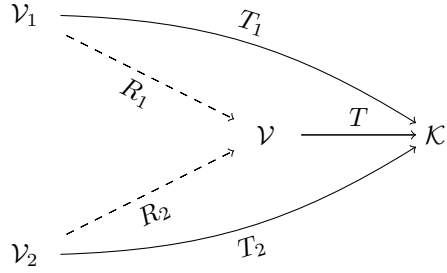
\begin{center}
\tikzset{
    electron/.style={draw=black, dashed, postaction={decorate},
        decoration={markings,mark=at position 0 with {\arrow[draw=black]{<}}}},
}
\tikzstyle{level 1}=[level distance=3.2cm, sibling distance=3.2cm]
\tikzstyle{level 2}=[level distance=3.2cm, sibling distance=1.8cm]
\tikzstyle{bag} = [text width=2em, text centered]

\begin{tikzpicture}[grow=left, sloped]

\node[bag] (a) {$\mc V$}
    child {
        node[bag] (c1) {$\mc V_1$}        
        edge from parent [electron]
            edge from parent 
            node[below] {$R_1$}
    }
    child {
        node[bag] (c2) {$\mc V_2$}        
        edge from parent [electron]
            edge from parent 
            node[below] {$R_2$}
    };
\draw[->,thin] (a) -- (2,0) node[right,scale=1.0] (b) {$\mc K$};     
\path[thin] (a) edge node[above] {$T$} (2,0);
\path (c1) edge [bend left=15,below,thin,->]  node[scale=1.0,above] {$T_1$} (b);
\path (c2) edge [bend right=15,below,thin,->]  node[scale=1.0,below] {$T_2$} (b);
\end{tikzpicture}
\end{center}

By $\Theta_j: (T_jT_j^*)' \ (\subseteq B(\mc K)) \to (T_j^*T_j)' \ (\subseteq B(\mc V_j)), \ j=1,2$, and
by $\Theta: (TT^*)' \ (\subseteq B(\mc K)) \to (T^*T)' \ (\subseteq B(\mc V))$ we shall denote the
$*$-algebra homomorphisms mapping the identity operator to the identity operator 
as in \thref{thetadefeig} from \cite{KaPr2014} corresponding to the mappings $T_j, \ j=1,2$, and $T$:
\[
    \Theta_j(C_j) = (T_j\times T_j)^{-1}(C_j) = T_j^{-1}C_jT_j, \ C_j\in (T_jT_j^*)' \,,
\]    
\begin{equation}\label{thetaVdef}
    \Theta(C) = (T\times T)^{-1}(C) = T^{-1}CT, \ C\in (TT^*)' \,.
\end{equation}
We can apply \thref{thetadefeig} in \cite{KaPr2014} also to the bounded linear, injective  
$R_j: \mc V_j \to \mc V, \ j=1,2$, and denote the corresponding $*$-algebra homomorphisms
by $\Gamma_j : (R_jR_j^*)' \ (\subseteq B(\mc V)) \to (R_j^*R_j)' \ (\subseteq B(\mc V_j))$:
\[
    \Gamma_j(D) = (R_j\times R_j)^{-1}(D) = R_j^{-1}DR_j, \ D \in (R_jR_j^*)' \,.
\]

\begin{proposition}\thlab{comreg}
    With the above notations and assumptions we have $(T_1T_1^*)' \cap (T_2T_2^*)' \subseteq (TT^*)'$ 
    and $\Theta((T_1T_1^*)' \cap (T_2T_2^*)') \subseteq (R_1R_1^*)' \cap (R_2R_2^*)' \cap (T^*T)'$,
    where in fact ($j=1,2$)
\begin{equation}\label{zuef}
  \Theta(C) R_j R_j^* = R_j \Theta_j(C) R_j^*=R_j R_j^* \Theta(C), \ \ C\in (T_1T_1^*)' \cap (T_2T_2^*)'  \,.
\end{equation}
    Moreover,
\begin{equation}\label{zuefvor}
  \Theta_j(C) = \Gamma_j \circ \Theta(C), \ \ C\in (T_1T_1^*)' \cap (T_2T_2^*)'  \,.
\end{equation}
\end{proposition}
\begin{proof}
    $(T_1T_1^*)' \cap (T_2T_2^*)' \subseteq (TT^*)'$ is clear from $TT^* = T_1T_1^* + T_2T_2^*$.
    According to \thref{thetadefeig} in \cite{KaPr2014} we have 
    $\Theta_j(C) T_j^* = T_j^* C$ and $T^* C = \Theta(C) T^*$ for $C\in (T_1T_1^*)' \cap (T_2T_2^*)'$.
    Therefore,
\begin{multline*}
    T (\ R_j \Theta_j(C) R_j^* \ ) T^* =  T_j \Theta_j(C) T_j^* = T_j T_j^* C = \\
    T R_j R_j^* T^* C = T ( \ R_j R_j^* \Theta(C) \ ) T^* \,. 
\end{multline*}
$\ker T=\{0\}$ and the density of $\ran T^*$ yield $R_j \Theta_j(C) R_j^*=R_j R_j^* \Theta(C)$
for $j=1,2$. Applying this equation to $C^*$ and taking adjoints yields 
$R_j \Theta_j(C) R_j^*=\Theta(C) R_j R_j^*$. In particular, $\Theta(C) \in (R_jR_j^*)'$.
Therefore, we can apply $\Gamma_j$ to $\Theta(C)$ and get
\[
  \Gamma_j \circ \Theta(C) = R_j^{-1} T^{-1}C T R_j = T_j^{-1} C T_j = \Theta_j(C) \,.
\]
\end{proof}

For the following assertion note that by \eqref{zuefvor} and by the fact that $\Gamma_j$
is a $*$-algebra homomorphism mapping the identity operator to the identity operator,
we have ($j=1,2$)
\begin{equation}\label{4ucd75}
  \sigma(\Theta(C)) \subseteq \sigma(\Theta_j(C)), \ C\in (T_1T_1^*)' \cap (T_2T_2^*)' \,.
\end{equation}

\begin{corollary}\thlab{normtransf}
     With the above notations and assumptions let $N\in (T_1T_1^*)' \cap (T_2T_2^*)'$ be normal.
     Then $\Theta(N),\Theta_1(N), \Theta_2(N)$ are all normal operators in the Hilbert spaces
     $\mc V$, $\mc V_1$, $\mc V_2$, respectively. If $E$ ($E_1$,$E_2$) denotes the spectral
     measure for $\Theta(N)$ ($\Theta_1(N)$, $\Theta_2(N)$), then $E(\Delta) \in (R_1R_1^*)' \cap (R_2R_2^*)' \cap (T^*T)'$ and
     \[
	\Gamma_j(E(\Delta)) = E_j(\Delta), \ \ j=1,2 \,,
     \]
     for all Borel subsets $\Delta$ of $\bb C$, where $E_j(\Delta) \in (R_j^*R_j)' \cap (T_j^*T_j)'$. 
     
     Moreover, $\int h \, dE \in (R_1R_1^*)' \cap (R_2R_2^*)' \cap (T^*T)'$ 
     and 
     \[
	\Gamma_j\left(\int h \, dE \right) = \int h \, dE_j 
     \]
     for any bounded and measurable $h: \sigma(\Theta(N)) \to \bb C$, where $\int h \, dE_j$
     belongs to $(R_j^*R_j)' \cap (T_j^*T_j)'$.
\end{corollary}
\begin{proof}
    The normality of $\Theta(N),\Theta_1(N)$ and $\Theta_2(N)$ is clear, since $\Theta$, $\Theta_1, \Theta_2$ are
    $*$-homomorphisms. From \thref{comreg} we know that $\Theta(N) \in (R_1R_1^*)' \cap (R_2R_2^*)' \cap (T^*T)'$.
    According to the well known properties of $\Theta(N)$'s spectral measure we obtain
    $E(\Delta) \in (R_1R_1^*)' \cap (R_2R_2^*)' \cap (T^*T)'$ and, in turn, 
    $\int h \, dE \in (R_1R_1^*)' \cap (R_2R_2^*)' \cap (T^*T)'$. In particular, $\Gamma_j$ can be applied 
    to $E(\Delta)$ and $\int h \, dE$.
    
    Similarly, $\Theta_j(N)\in (T_j^*T_j)'$ implies $E_j(\Delta), \int h \, dE_j \in (T_j^*T_j)'$ 
    for a bounded and measurable $h$.
    
    Recall from \thref{thetadefeig} in \cite{KaPr2014} that $\Gamma_j(D) R_j^*x = R_j^* D$
    for $D \in (T^*T)'$. For $x\in \mc V$ and $y\in \mc V_j$ we therefore get
    \[
	[\Gamma_j(E(\Delta)) R_j^*x, y] = [R_j^* E(\Delta) x,y] = [E(\Delta) x, R_j y]
    \]
    and, in turn,
    \begin{multline*}
	\int_{\bb C} s(z,\bar z) \, d[\Gamma_j(E) R_j^*x, y] = 
	\int_{\bb C} s(z,\bar z) \, d[E x, R_j y] = [s(\Theta(N),\Theta(N)^*)x, R_j y] = \\
	[R_j ^* s(\Theta(N),\Theta(N)^*)x, y] = [\Gamma_j(s(\Theta(N),\Theta(N)^*)) R_j^* x,y] 
    \end{multline*}
    for any trigonometric polynomial $s(z,\bar z) \in \bb C[z,\bar z]$.
    By \eqref{zuefvor} and the fact that $\Gamma_j$ is a $*$-homomorphism 
    we have $\Gamma_j(s(\Theta(N),\Theta(N)^*)) = s(\Theta_j(N),\Theta_j(N)^*)$.
    Consequently,
    \[
	\int_{\bb C} s(z,\bar z) \, d[\Gamma_j(E) R_j^*x, y] = \int_{\bb C} s(z,\bar z) \, d[E_j R_j^*x, y] \,.
    \]
    Since $E(\bb C\setminus K)=0$ and $E_j(\bb C\setminus K)=0$ for a certain compact $K\subseteq \bb C$ and since 
    $\bb C[z,\bar z]$ is densely contained in $C(K)$, we obtain from the uniqueness assertion of 
    the Riesz Representation Theorem 
    \[
	[\Gamma_j(E(\Delta)) R_j^*x, y] = [E_j(\Delta) R_j^*x, y], \ x\in \mc V, \, y\in \mc V_j \,,
    \]
    for all Borel subsets $\Delta$ of $\bb C$.
    Due to the density of $\ran R_j^*$ in $\mc V_j$ we even have 
    $[\Gamma_j(E(\Delta)) z, y] = [E_j(\Delta) z, y], \ y,z \in \mc V_j$ and, in turn,
    $\Gamma_j(E(\Delta))=E_j(\Delta)$. Since $\Gamma_j$ maps into $(R_j^*R_j)'$ we have
    $E_j(\Delta)\in (R_j^*R_j)'$ and, in turn, $\int h \, dE_j\in (R_j^*R_j)'$ for any bounded and measurable $h$.
    
    If $h: \sigma(\Theta(N)) \to \bb C$ is bounded and measurable, then, clearly, also
    its restriction to $\sigma(\Theta_j(N)) = \sigma(\Gamma_j\circ\Theta(N))$ is bounded and measurable;
    see \eqref{4ucd75}. 
    Due to $E_j(\Delta) R_j^*= \Gamma_j(E(\Delta)) R_j^* = R_j^*E(\Delta)$ for 
    $x\in\mc V$ and $y\in\mc V_j$ we have
    \begin{multline*}
	[\Gamma_j\left(\int h \, dE \right) R_j^*x,y] = 
	[R_j^* \left(\int h \, dE \right) x,y] = 
	[\left(\int h \, dE \right) x, R_j y] = \\
	\int h \, d[E x,R_j y] = \int h \, d[E_j R_j^* x, y] =
	[\left(\int h \, dE_j\right) R_j^*x,y] \,. 
    \end{multline*}
    Again the density of $\ran R_j^*$ yields $\Gamma_j\left(\int h \, dE \right) = \int h \, dE_j$.
\end{proof}

Recall from \thref{Xidefeig} in \cite{KaPr2014} the mappings ($j=1,2$)
\begin{equation}\label{xijdef}
    \Xi_j : (T_j^*T_j)' \ (\subseteq B(\mc V_j)) \to (T_jT_j^*)' \ (\subseteq B(\mc K)), \ 
		\Xi_j(D_j) = T_j D_j T_j^{*} \,,
\end{equation}
and $\Xi : (T^*T)' \ (\subseteq B(\mc V)) \to (TT^*)' \ (\subseteq B(\mc K)), \ \Xi(D) = T D T^{*}$.
By ($j=1,2$)
\[
    \Lambda_j: (R_j^*R_j)' \ (\subseteq B(\mc V_j)) \to (R_jR_j^*)' \ (\subseteq B(\mc V)), 
	\ \Lambda_j(D_j) = R_j D_j R_j^{*} \,,
\]
we shall denote the corresponding mappings outgoing from the mappings $R_j: \mc V_j \to \mc V$. 
By \thref{existtreans} we have 
\[
    \Xi_j(D_j) = T_j R_j D_j R_j^* T_j^* = \Xi\circ\Lambda_j(D_j) \ \text{ for } 
	      \ D_j\in (R_j^*R_j)' \cap (T_j^*T_j)' \,.
\]
According to \thref{Xidefeig} in \cite{KaPr2014}, $\Lambda_j\circ \Gamma_j(D) = D R_jR_j^*$.
Hence, using the notation from \thref{normtransf}
\begin{equation}\label{zuef2}
    \Xi_j(\int h \, dE_j) = \Xi\circ\Lambda_j\circ \Gamma_j\left(\int h \, dE \right) = \Xi(R_jR_j^{*} \int h \, dE) \,.
\end{equation}
Finally, $T^{-1} T_jT_j^* T = T^{-1} T R_j R_j^* T^* T = R_j R_j^* T^* T$. In case that 
$T_1T_1^*$ and $T_2T_2^*$ commute we have $T_1T_1^*, T_2T_2^* \in (TT^*)'$ and the later equality can be
expressed as ($j=1,2$)
\begin{equation}\label{zuef3}
    \Theta(T_jT_j^*) = R_j R_j^* T^* T \,.
\end{equation}
% 
% In particular, $R_1 R_1^*\Theta(T_2T_2^*) = R_j R_j^* T^* T = R_2 R_2^*\Theta(T_1T_1^*)$.
% 

\section{Normal definitizable operators}

\begin{definition}\thlab{definitnordef}
We will call a bounded linear and normal operator $N$ on a Krein Space 
\emph{definitizable} if its real part $A:=\frac{N+N^*}{2}$ and its 
imaginary part $B:=\frac{N-N^*}{2i}$ are both definitizable, i.e.\ there exist
real polynomials $p,q\in\bb R[z]$ such that $p$ is definitizing for $A$
($[p(A)x,x] \geq 0, \, x\in \mc K$) and such that $q$ is definitizing for $B$
($[q(B)x,x] \geq 0, \, x\in \mc K$); see \cite{langer1982}.
\end{definition}

By \thref{anderedefinitz} in \cite{KaPr2014} the definitizability of $A$ and $B$
is equivalent to the concept of definitizability in \cite{KaPr2014}.

Also note that in Pontryagin spaces any bounded linear and normal operator is
definitizable in the above sense; see \thref{pontrunit} in \cite{KaPr2014}.

\begin{proposition}\thlab{defNspaces}
  Let $A$ and $B$ be commuting, bounded linear, selfadjoint and definitizable operators on a Krein space
  $(\mc K,[.,.])$ with definitizing polynomials $p\in\bb R[z]$ for $A$ and $q\in\bb R[z]$ for $B$.
  Then there exist Hilbert spaces $(\mc V_1,[.,.])$, $(\mc V_2,[.,.])$, $(\mc V,[.,.])$ and
  bounded linear and injective operators $T_1: \mc V_1 \to \mc K$, $T_2: \mc V_2 \to \mc K$, $T: \mc V \to \mc K$
  such that
  \[
      T_1 T_1^* = p(A), \ T_2 T_2^* = q(B), \ T T^* = p(A)+q(B) = T_1 T_1^* + T_2 T_2^*
  \]
  with commuting $T_1 T_1^*$ and $T_2 T_2^*$.
  Moreover, if $\Theta: (TT^*)' \ (\subseteq B(\mc K)) \to (T^*T)' \ (\subseteq B(\mc V))$ is as in
  \eqref{thetaVdef} and $R_j: \mc V_j \to \mc V$ ($j=1,2$) are as in \thref{existtreans}, then
\begin{equation}\label{heaybab}
\begin{aligned}
    p(\Theta(A)) = R_1 R_1^* \big(p(\Theta(A)) + q(\Theta(B)) \big), \\
    q(\Theta(B)) = R_2 R_2^* \big(p(\Theta(A)) + q(\Theta(B)) \big)
    \,,
\end{aligned}
\end{equation}
    where $R_1 R_1^*$ and $R_2 R_2^*$ commute with $p(\Theta(A)) + q(\Theta(B))$.
\end{proposition}
\begin{proof}
Let $(\mc V_1,[.,.])$ be the Hilbert space completion of $\mc K/\ker p(A)$ with respect to 
$[p(A).,.]$ and let $T_1: \mc V_1 \to \mc K$
be the adjoint of the embedding of $\mc K$ into $\mc V_1$. Since $T_1^*$ has dense range,
$T_1$ is injective.
Analogously let 
$(\mc V_2,[.,.])$ be the Hilbert space completion of $\mc K/\ker q(B)$ with respect to 
$[q(B).,.]$ and denote by $T_2:\mc V_2 \to \mc K$
the injective adjoint of the embedding of $\mc K$ into $\mc V_2$. 
Finally, let 
$(\mc V,[.,.])$ be the Hilbert space completion of $\mc K/(\ker p(A)+q(B))$ with respect 
to $[(p(A)+q(B)).,.]$ and let
$T: \mc V \to \mc K$ be the injective adjoint of the embedding of $\mc K$ into $\mc V$. 

From $[T T^* x,y] = [T^* x,T^* y]_{\mc V} = [x,y]_{\mc V} = [(p(A)+q(B))x,y]$,  
$[T_1 T_1^* x,y] = [T_1^* x,T_1^* y]_{\mc V_1} = [x,y]_{\mc V_1} = [p(A)x,y]$ and 
$[T_2 T_2^* x,y] = [q(B)x,y]$ for all $x,y\in\mc K$ we conclude that
  \[
      T_1 T_1^* = p(A), \ T_2 T_2^* = q(B), \ T T^* = p(A)+q(B) \,,
  \]
where $p(A)=T_1T_1^*$ and $q(B)=T_2T_2^*$ commute, because $A$ and $B$ do.

From \eqref{zuef3} and \thref{thetadefeig} in \cite{KaPr2014} we get
\begin{multline*}
      p(\Theta(A)) = \Theta(p(A)) = \Theta(T_1T_1^*) = R_1 R_1^* T^* T = R_1 R_1^* \Theta(TT^*) = \\
      R_1 R_1^* \Theta( p(A)+q(B) ) =  R_1 R_1^* \big(p(\Theta(A)) + q(\Theta(B)) \big) \,. 
\end{multline*}
Similarly, $q(\Theta(B)) = R_2 R_2^* (p(\Theta(A)) + q(\Theta(B)) )$.
Finally, $R_1 R_1^*$ and $R_2 R_2^*$ commute with $T^* T=p(\Theta(A)) + q(\Theta(B))$ by \thref{existtreans}.
\end{proof}

The fact that a normal operator is definitizable implies certain spectral properties
of $\Theta(N)$.

\begin{lemma}\thlab{speknorm}
With the notion of \thref{defNspaces} applied to the real part $A:=\frac{N+N^*}{2}$ and the
imaginary part $B:=\frac{N-N^*}{2i}$ of a bounded linear, normal and definitizable operator $N$
we have
\[
    \{z\in \bb C : |p(\Re z)| > \| R_1 R_1^* \| \cdot |p(\Re z) + q(\Im z)| \} \subseteq \rho(\Theta(N)) \,,
\]
and 
\[
    \{z\in \bb C : |q(\Im z)| > \| R_2 R_2^* \| \cdot |p(\Re z) + q(\Im z)| \} \subseteq \rho(\Theta(N)) \,.
\]
In particular, the zeros of $p(\Re z) + q(\Im z)$ are contained in 
$\rho(\Theta(N)) \cup \{z\in \bb C: p(\Re z) = 0 = q(\Im z)\}$.
\end{lemma}
\begin{proof}
We are going to show the first inclusion. The second one is shown in the same manner.
For this let $n\in \bb N$ and set 
\[
    \Delta_n:= \{z\in \bb C : |p(\Re z)|^2 > \frac{1}{n} + \| R_1 R_1^* \|^2 \cdot |p(\Re z) + q(\Im z)|^2 \} \,.
\]
For $x\in E(\Delta_n)(\mc V)$ we then have
\begin{multline*}
    \| p(\Theta(A)) x \|^2 = \int_{\Delta_n} |p(\Re \zeta)|^2 \, d[E(\zeta)x,x] \geq \\ \int_{\Delta_n} \frac{1}{n} \, d[E(\zeta)x,x] +
	\| R_1 R_1^* \|^2 \int_{\Delta_n} |p(\Re \zeta) + q(\Im \zeta)|^2 \, d[E(\zeta)x,x] \\ \geq
	\frac{1}{n} \|x\|^2 + \| R_1 R_1^* \big(p(\Theta(A)) + q(\Theta(B))x \|^2 \,.
\end{multline*}
By \eqref{heaybab} this inequality can only hold for $x=0$. Since $\Delta_n$ is open,
by the Spectral Theorem for normal operators on Hilbert spaces we have $\Delta_n \subseteq \rho(N)$.
The asserted inclusion now follows from 
\[
  \{z\in \bb C : |p(\Re z)| > \| R_1 R_1^* \| \cdot |p(\Re z) + q(\Im z)| \} = \bigcup_{n\in\bb N} \Delta_n \,.
\]
\end{proof}

\begin{corollary}\thlab{korvda}
    With the notation and assumptions from \thref{speknorm} we have
\begin{multline*}
	R_1R_1^* \, E\{z\in \bb C: p(\Re z) \not= 0 \text{ or } q(\Im z)\not=0\} =  \\
	\int_{\{z\in \bb C: p(\Re z) \not= 0 \text{ or } q(\Im z)\not=0\}} \frac{p(\Re z)}{p(\Re z) + q(\Im z)} \, dE(z) 
\end{multline*}
    and
\begin{multline*}
	R_2R_2^* \, E\{z\in \bb C: p(\Re z) \not= 0 \text{ or } q(\Im z)\not=0\} =  \\
	\int_{\{z\in \bb C: p(\Re z) \not= 0 \text{ or } q(\Im z)\not=0\}} \frac{q(\Re z)}{p(\Re z) + q(\Im z)} \, dE(z) 
\end{multline*}
\end{corollary}
\begin{proof}
    First note that the integrals on the right hand sides exist as bounded operators, because 
    by \thref{speknorm} we have $|p(\Re z)| \leq \| R_1 R_1^* \| \cdot |p(\Re z) + q(\Im z)|$ and
    $|q(\Im z)| \leq \| R_2 R_2^* \| \cdot |p(\Re z) + q(\Im z)|$ on $\sigma(\Theta(N))$.
    
    Clearly, both sides vanish on the range of $E\{z\in \bb C: p(\Re z) = 0 = q(\Im z)\}$. 
    Its orthogonal complement 
\begin{multline*}
	\mc H:=\ran E\{z\in \bb C: p(\Re z) = 0 = q(\Im z)\}^\bot = \\
	  \ran E\{z\in \bb C: p(\Re z) \not= 0 \text{ or } q(\Im z)\not=0\} \,, 
\end{multline*}
    is invariant under
    $\int \big(p(\Re z) + q(\Im z)\big) \, dE(z) = \big(p(\Theta(A)) + q(\Theta(B))\big)$. 
    By \thref{speknorm} the restriction of this operator to $\mc H$
    is injective, and hence, has dense range in $\mc H$.
    If $x$ belongs to this dense range, i.e.\ 
    $x=\big(p(\Theta(A)) + q(\Theta(B))\big) y$ with $y\in \mc H$, then
    \[
	\int_{\{z\in \bb C: p(\Re z) \not= 0 \text{ or } q(\Im z)\not=0\}} 
			      \frac{p(\Re z)}{p(\Re z) + q(\Im z)} \, dE(z) x = 
    \]
\begin{multline*}
	\int_{\{z\in \bb C: p(\Re z) \not= 0 \text{ or } q(\Im z)\not=0\}} p(\Re z) \, dE(z) y =
	p(\Theta(A))y = \\ R_1 R_1^* \big(p(\Theta(A)) + q(\Theta(B))\big) y = R_1 R_1^* x \,.
\end{multline*}
    By a density argument the first asserted equality of the present corollary holds true on $\mc H$
    and in turn on $\mc V$. 
    The second equality is shown in the same manner.
\end{proof}

\section{The proper function class}

In order to introduce a functional calculus 
we have to introduce an algebra structure on 
$\mc A_{m,n}:=(\mathbb C^{m} \otimes \mathbb C^{n}) \times \mathbb C^2 \simeq \mathbb C^{m\cdot n + 2}$
and on $\mc B_{m,n} := \mathbb C^{m} \otimes \mathbb C^{n} \simeq \mathbb C^{m\cdot n}$
for $m,n\in \bb N$. For notational
convenience we also set $\mc A_{0,0} := \mathbb C$.

\begin{definition}\thlab{muldefb1}
    Firstly, let $\mc A_{0,0} = \bb C$ be provided with the usual addition, scalar multiplication, 
    multiplication and conjugation.
    
    Secondly, in case that $m,n\in \bb N$ we provide $\mc A_{m,n}$ with the componentwise
    addition and scalar multiplication. Moreover, for 
    $a=(a_{k,l})_{(k,l)\in I_{m,n}}, b=(b_{k,l})_{(k,l)\in I_{m,n}}$ with 
    $I_{m,n}:=(\{0,\dots,m-1\}\times \{0,\dots,n-1\})\cup\{(m,0),(0,n)\}$ we set
\[
	a\cdot b := 
	\Big( \sum_{c=0}^k\sum_{d=0}^l a_{c,d} b_{k-c,l-d} \Big)_{(k,l)\in I_{m,n}} \ \text{ and } \
	\overline{a}:=\big(\bar a_{k,l}\big)_{(k,l)\in I_{m,n}} \,.
\]
    On $\mc B_{m,n}$ we define addition, scalar multiplication, multiplication and conjugation in the same way 
    only neglecting the the entries with indices $(m,0)$ and $(0,n)$.

Finally, for $m,n\in \bb N$ we introduce the projection 
$\pi: \mc A_{m,n} \to \mc B_{m,n}$, 
$(a_{k,l})_{(k,l)\in I_{m,n}} \mapsto (a_{k,l})_{\substack{0\leq k \leq m-1 \\ 0\leq l\leq n-1}}$.
On $\mc B_{m,n}$ we assume $\pi$ to be the identity. 
\end{definition}

\begin{remark}\thlab{z30f3}
It is easy to check that $\mc A_{m,n}$ and $\mc B_{m,n}$ are commutative, unital $*$-algebras.
Setting $e_{0,0}=1$ and $e_{k,l}=0, \ (k,l)\neq (0,0)$, it is easy to verify that
$\big(e_{k,l}\big)_{(k,l)\in I_{m,n}}$ is the multiplicative unite in $\mc A_{m,n}$ 
and $\big(e_{k,l}\big)_{\substack{0\leq k \leq m-1 \\ 0\leq l\leq n-1}}$ is the 
multiplicative unite in $\mc B_{m,n}$. We shall denote these unites by $e$.

Moreover, it is straight forward to check that an element $(a_{k,l})$ of $\mc A_{m,n}$ 
(of $\mc B_{m,n}$)
has a multiplicative inverse in $\mc A_{m,n}$ (in $\mc B_{m,n}$) if and only if $a_{0,0}\neq 0$.
\end{remark}

For the rest of the paper assume that $N$ bounded linear, normal and definitizable operator
in a Krein space $\mc K$ with real part $A$ and imaginary part $B$.
Moreover, we fix definitizing polynomials $p\in\bb R[z]$ for $A$ and 
$q\in\bb R[z]$ for $B$.

\begin{definition}\thlab{muldefb2}
    We define functions $\mf d_p, \mf d_q: \bb C \to \bb N\cup\{0\}$ 
    such that $\mf d_p(z)$ is $p$'s degrees of the zero at $z$ and 
    $\mf d_q(z)$ is $q$'s degrees of the zero at $z$.
    Moreover, we shall denote the set of their real zeros by $Z^{\bb R}_p$ and $Z^{\bb R}_q$, i.e.\
    \[
	Z^{\bb R}_p:=p^{-1}\{0\}\cap \bb R, Z^{\bb R}_q:=q^{-1}\{0\}\cap \bb R \,,
    \]
    and we set $Z^i:=(p^{-1}\{0\}\times q^{-1}\{0\})\setminus (\bb R\times \bb R)$.
    
    \noindent
    Now we are going to introduce class of functions:
    \begin{enumerate}[$(i)$]
    \item
    By $\mc M_{N}$ we denote the set of functions $\phi$ defined on 
    \[
      \big(\sigma(\Theta(N)) \cup (Z^{\bb R}_p + i Z^{\bb R}_q)\big)\dot\cup Z^i 
    \]
    with $\phi(z) \in \mf C(z)$, where $\mf C(z):=\mc B_{\mf d_p(\xi),\mf d_q(\eta)}$ for $z=(\xi,\eta) \in Z^i$
    and where
    $\mf C(z):=\mc A_{\mf d_p(\Re z),\mf d_q(\Im z)}$
    for $z \in \sigma(\Theta(N)) \cup (Z^{\bb R}_p + i Z^{\bb R}_q)$.
    \item
    We provide $\mc M_{N}$ pointwise with scalar multiplication, 
    addition and multiplication, where the operations on $\mc A_{\mf d_p(\Re z),\mf d_q(\Im z)}$
    or $\mc B_{\mf d_p(\xi),\mf d_q(\eta)}$ are as in \thref{muldefb1}.
    We also define a conjugate linear involution $.^\#$ on $\mc M_{N}$ by 
\begin{multline*}
	\phi^\#(z) = \overline{\phi(z)}, \ z \in \sigma(\Theta(N)) \cup (Z^{\bb R}_p + i Z^{\bb R}_q), \\
	\phi^\#(\xi,\eta) = \overline{\phi(\bar \xi,\bar \eta)}, \ (\xi,\eta) \in Z^i \,.
\end{multline*}
    \item
    By $\mc R$ we denote the set of all elements $\phi\in\mc M_{N}$ such that
    $\pi(\phi(z)) = 0$ for all $z \in (Z^{\bb R}_p + i Z^{\bb R}_q) \dot \cup Z^i$.
\end{enumerate}
\end{definition}

With the operations introduced in \thref{muldefb2} $\mc M_{N}$ is a commutative $*$-algebra as can be verified 
in a straight forward manner. Moreover, $\mc R$ is an ideal of $\mc M_{N}$.    
    
\begin{definition}\thlab{feinbetef}
    Let $f: \dom f \to \bb C$ be a function with $\dom f \subseteq \bb C^2$ such that
    $\tau\big(\sigma(\Theta(N)) \cup (Z^{\bb R}_p + i Z^{\bb R}_q)\big) \subseteq \dom f$, where 
    $\tau : \bb C \to \bb C^2, \ (x+iy)\mapsto (x,y)$, such that
    $f\circ \tau$ is sufficiently smooth -- more exactly, at least 
    $\max_{x,y\in \bb R} \mf d_p(x) + \mf d_q(y) - 1$ times continuously differentiable -- 
    on an open neighbourhood of $Z^{\bb R}_p + i Z^{\bb R}_q$, and
    such that $f$ is holomorphic on an open neighbourhood of $Z^i$.

    Then $f$ can be considered as an element $f_{N}$ of $\mc M_{N}$ by setting
    $f_{N}(z) := f\circ \tau(z)$ for $z\in \sigma(\Theta(N)) \setminus (Z^{\bb R}_p + i Z^{\bb R}_q)$, by
\[
    f_{N}(z) := 
    \big(\frac{1}{k!l!} \, \frac{\partial^{k+l}}{\partial x^k\partial y^l} 
	f\circ \tau (z)\big)_{(k,l)\in I_{\mf d_p(\Re z),\mf d_q(\Im z)}}
\]
    for $z\in Z^{\bb R}_p + i Z^{\bb R}_q$, and by 
    \[
	f_{N}(\xi,\eta) := \big(\frac{1}{k!l!} \, \frac{\partial^{k+l}}{\partial z^k\partial w^l} 
		f(\xi,\eta)\big)_{\substack{0\leq k \leq \mf d_p(\xi)-1 \\ 0\leq l\leq \mf d_q(\eta) - 1}} \,,
    \]
    for $(\xi,\eta)\in Z^i$.
\end{definition}

\begin{remark}\thlab{bweuh30}
By the Leibniz rule $f\mapsto f_{N}$ is compatible with multiplication. Obviously, it is also compatible with 
addition and scalar multiplication. If we define for a function $f$ as in \thref{feinbetef} the function 
$f^\#$ by $f^\#(z,w) = \overline{f(\bar z,\bar w)}, \ (z,w) \in \dom f$, then we also have 
$(f^\#)_{p,q} = (f_{N})^\#$. Note that in general $(\bar f)_{p,q} \not= (f_{N})^\#$.

Finally, note that $\mathds{1}_N(z)$ is the multiplicative unite in $\mf C(z)$ for all 
$z\in \big(\sigma(\Theta(N)) \cup (Z^{\bb R}_p + i Z^{\bb R}_q)\big)\dot\cup Z^i$.
\end{remark}

\begin{example}\thlab{fuinm0pre}
    For the constant one function $\mathds{1}$ on $\bb C^2$ we have $\mathds{1}_N(z) = e$
    for all $z\in \big(\sigma(\Theta(N)) \cup (Z^{\bb R}_p + i Z^{\bb R}_q)\big)\dot\cup Z^i$, where
    $e$ is the multiplicative unite in $\mf C(z)$; see \thref{z30f3}. 
\end{example}

\begin{example}\thlab{fuinm0}
 $p(z)$ considered as an element of $\bb C[z,w]$ is clearly holomorphic on $\bb C^2$.
 Hence, we can consider $p_N$ as defined in \thref{feinbetef}.
 It satisfies $p_{N}(z)_{k,l}=0$,  
 $(k,l)\in I_{\mf d_p(\Re z),\mf d_q(\Im z)}\setminus \{(\mf d_p(\Re z),0)\}$, and
 \[
      p_{N}(z)_{\mf d_p(\Re z),0} = \frac{1}{\mf d_p(\Re z)!} p^{(\mf d_p(\Re z))}(\Re z)
 \]
 for all $z \in Z^{\bb R}_p + i Z^{\bb R}_q$. 
 Since $\Re z$ is a zero of $p$ of degree
 exactly $\mf d_p(\Re z)$ the entries with index $(\mf d_p(\Re z),0)$ do not vanish. 
 Moreover, $p_{N}(\xi,\eta)=0$ for all $(\xi,\eta)\in Z^i$. In particular, $p_{N} \in \mc R$.
 
 Similarly, if $q(w)$ is considered as an element of $\bb C[z,w]$, then $q_{N}(z)_{k,l}=0$,  
 $(k,l)\in I_{\mf d_p(\Re z),\mf d_q(\Im z)}\setminus \{(0,\mf d_q(\Im z))\}$, and
 \[
      q_{N}(z)_{0,\mf d_q(\Im z)} = \frac{1}{\mf d_q(\Im z)!} q^{(\mf d_q(\Im z))}(\Im z) \neq 0
 \]
 for all $z \in Z^{\bb R}_p + i Z^{\bb R}_q$. 
 Also here $q_{N}(\xi,\eta)=0$ for all $(\xi,\eta)\in Z^i$ and, in turn, $q_{N} \in \mc R$.
\end{example}

We need an easy algebraic lemma based in the Euclidean algorithm.

\begin{lemma}\thlab{einbett2pre}
    For $a(z), b(z)\in \bb C[z]$ we denote by $a^{-1}\{0\}$ and $b^{-1}\{0\}$ the set of all zeros of
    $a$ and $b$ in $\bb C$, and by $\mf d_a(z)$ ($\mf d_b(z)$) $a$'s ($b$'s) 
    degree of zero at $z\in \bb C$. Denote by $m$ ($n$) the degree of the polynomial $a$ ($b$).
    Then any $s\in \bb C[z,w]$ can be written as 
    \[
	s(z,w) = a(z)u(z,w) + b(w)v(z,w) + r(z,w)
    \]
    with $u(z,w), v(z,w), r(z,w)\in \bb C[z,w]$
    such that $r$'s $z$-degree is less than $m$ and its $w$-degree is less than $n$.
    Here $u(z,w), v(z,w), r(z,w)$ can be found in $\bb R[z,w]$ if $a(z), b(z)\in \bb R[z], \ s\in \bb R[z,w]$.
    
    If we define $\varpi: \bb C[z,w] \to \bb C^{m\cdot n}$ by
    \[
	\varpi(s) = \left( \Big(\frac{1}{k!l!} \, \frac{\partial^{k+l}}{\partial z^k\partial w^l} 
	s(z,w)\Big)_{\substack{0\leq k \leq \mf d_a(z)-1 \\ 0\leq l\leq \mf d_b(w)-1}}\right)_{
	z\in a^{-1}\{0\}, w\in b^{-1}\{0\}} \,,
    \]
    then $s\in \ker \varpi$ if and only if $s(z,w)=a(z)u(z,w) + b(w) v(z,w)$
    for some $u(z,w), v(z,w) \in \bb C[z,w]$. Moreover, $\varpi$ restricted to the space of all
    polynomials from $\bb C[z,w]$ with $z$-degree less than $m$ and $w$-degree less than $n$
    is bijective.
\end{lemma}
\begin{proof}
    Applying the Euclidean algorithm to $s(z,w)\in \bb C[z,w]$ and $a(z)$ we get
    $s(z,w) = a(z)u(z,w) + t(z,w)$, where $u(z,w), t(z,w)\in \bb C[z,w]$ such that
    $t$'s $z$-degree is less than $m$. Applying the Euclidean algorithm to $t(z,w)$ and $b(w)$ we get
    \[
	s(z,w) = a(z)u(z,w) + b(w)v(z,w) + r(z,w) 
    \]
    with $v(z,w), r(z,w)\in \bb C[z,w]$
    such that $r$'s $z$-degree is less than $m$ and its $w$-degree is less than $n$.
    The resulting polynomials $u(z,w), t(z,w), v(z,w), r(z,w)$ belong to $\bb R[z,w]$ if
    $a(z), b(z)\in \bb R[z], \ s(z,w)\in \bb R[z,w]$.
    
    In any case it is easy to check that then $\varpi(s)=\varpi(r)$. 
    Hence, $r(z,w)=0$ yields $s(z,w)\in \ker \varpi$. On the other hand, if $0=\varpi(s)=\varpi(r)$, then
    for each fixed $\zeta \in a^{-1}\{0\}$ and $k\in\{0,\dots, \mf d_a(\zeta)-1\}$ the function 
    $w\mapsto \frac{\partial^{k}}{\partial z^k} r(\zeta,w)$ has zeros at all 
    $w\in b^{-1}\{0\}$ with multiplicity at least $\mf d_b(w)$. 
    Since $w\mapsto \frac{\partial^{k}}{\partial z^k} r(\zeta,w)$ is of 
    $w$-degree less than $n$, it must be identically equal to zero.
    
    This implies that for any $\eta \in \bb C$ the polynomial $z \mapsto r(z,\eta)$ has zeros at all 
    $\zeta\in a^{-1}\{0\}$ with multiplicity at least $\mf d_a(\zeta)$. Since the degree of this polynomial in $z$
    is less than $m$, we obtain $r(z,\eta)=0$ for any $z\in \bb C$. Thus, $r\equiv 0$.

    Our description of $\ker \varpi$ shows in particular that $\varpi$ restricted to the space of all
    polynomials from $\bb C[z,w]$ with $z$-degree less than $m$ and $w$-degree less than $n$
    is one-to-one. Comparing dimensions shows that this restriction of $\varpi$ is also onto.
\end{proof}

\begin{corollary}\thlab{existbe}
    With the notation from \thref{muldefb2} 
    for any $\phi \in \mc M_{N}$ we find an 
    $s\in \bb C[z,w]$ such that $\phi - s_N \in \mc R$.
\end{corollary}
\begin{proof}
    By \thref{einbett2pre} there exists an $s\in \bb C[z,w]$ such that 
    $\varpi(s)_{\Re z,\Im z} = \pi(\phi(z))$ for all 
    $z \in Z^{\bb R}_p + i Z^{\bb R}_q$, and 
    such that $\varpi(s)_{\xi,\eta} = \phi(\xi,\eta)$ for all $(\xi,\eta) \in Z^i$.
    According to $\mc R$'s definition we obtain $\phi - s_N \in \mc R$. 
\end{proof}

\begin{remark}\thlab{durchdiv}
    Recall from \thref{speknorm} that $p(\Re z) + q(\Im z)=0$ with $z\in \sigma(\Theta(N))$ 
    implies $p(\Re z) = 0 = q(\Im z)$, i.e.\ $z\in Z^{\bb R}_p + i Z^{\bb R}_q$. 
    
    If $\phi\in \mc R$, then we find a function $g$ on $\sigma(\Theta(N))$ with 
    $g(z)\in \bb C$ for $z\in \sigma(\Theta(N))\setminus (Z^{\bb R}_p + i Z^{\bb R}_q)$ and
    $g(z)\in \bb C^2$ for $z\in \sigma(\Theta(N))\cap (Z^{\bb R}_p + i Z^{\bb R}_q)$, such that 
    $\phi(z) = (p_N + q_N)(z)\cdot g(z), \ z \in \sigma(\Theta(N))$; see \thref{fuinm0}. 
    Here $(p_N + q_N)(z) \cdot g(z)$
    is the usual multiplication on $\bb C$ for $z\in \sigma(\Theta(N))\setminus (Z^{\bb R}_p + i Z^{\bb R}_q)$, 
    whereas
    \[
    \big((p_N + q_N)(z) \cdot g(z)\big)_{k,l} = 0, \ k = 0,\dots,\mf d_p(\Re z)-1; l = 0,\dots,\mf d_q(\Im z)-1 \,,
    \]
    and 
    \[
    \big((p_N + q_N)(z)\cdot g(z)\big)_{\mf d_p(\Re z),0} 
	= (p_N + q_N)(z)_{\mf d_p(\Re z),0} \, \cdot g_1(z) \,,
    \]\[
	\big((p_N + q_N)(z)\cdot g(z)\big)_{0,\mf d_q(\Im z)} 
	= (p_N + q_N)(z)(z)_{0,\mf d_q(\Im z)} \, \cdot g_2(z) \,.
   \]
    for $z\in \sigma(\Theta(N))\cap (Z^{\bb R}_p + i Z^{\bb R}_q)$.
    
    In fact, we simply set $g(z):=\frac{\phi(z)}{p(\Re z)+q(\Im z)}$ for
    $z\in \sigma(\Theta(N))\setminus (Z^{\bb R}_p + i Z^{\bb R}_q)$ and
    \[
	g_1(z) := \frac{\mf d_p(\Re z)! \, \phi(z)_{(\mf d_p(\Re z),0)}}{p^{(\mf d_p(\Re z))}(\Re z)}, \
	g_1(z) := \frac{\mf d_q(\Im z)! \, \phi(z)_{(0,\mf d_q(\Im z))}}{q^{(\mf d_q(\Im z))}(\Im z)}
    \]
    for $z\in \sigma(\Theta(N))\cap (Z^{\bb R}_p + i Z^{\bb R}_q)$.   
\end{remark}

We are going to introduce a subclass of $\mc M_{N}$, which will be the proper class, in order to build up our functional
calculus.

\begin{definition}\thlab{FdefklM}
    With the notation from \thref{muldefb2} we denote by $\mc F_{N}$ the set of all elements 
    $\phi \in \mc M_{N}$ such that
    $z\mapsto \phi(z)$ is Borel measurable and bounded on 
    $\sigma(\Theta(N)) \setminus (Z^{\bb R}_p + i Z^{\bb R}_q)$, and such that
    for each $w \in \sigma(\Theta(N)) \cap (Z^{\bb R}_p + i Z^{\bb R}_q)$ 
     \begin{equation}\label{fn8qw3b}
    \frac{\phi(z) - 
		\sum_{k=0}^{\mf d_p(\Re w)-1}\sum_{l=0}^{\mf d_q(\Im w)-1} 
		\phi(w)_{k,l} \Re(z-w)^k \Im(z-w)^l}{\max(|\Re(z-w)|^{\mf d_p(\Re w)},|\Im (z-w)|^{\mf d_q(\Im w)})}
     \end{equation}
    is bounded for $z\in \sigma(\Theta(N))\cap U(w)\setminus \{w\}$, 
    where $U(w)$ is a sufficiently small neighbourhood of $w$.
\end{definition}

Note that \eqref{fn8qw3b} is immaterial if $w$ is an isolated point of $\sigma(\Theta(N))$.

\begin{example}\thlab{fedela}
    For $\zeta \in (Z^{\bb R}_p + i Z^{\bb R}_q)\dot\cup Z^i$ and $a\in \mf C(\zeta)$ consider the functions
    $a\delta_\zeta \in \mc M_{N}$ which assumes the value $a$ at $\zeta$ and the value zero on the rest of
    $\big(\sigma(\Theta(N)) \cup (Z^{\bb R}_p + i Z^{\bb R}_q)\big)\dot\cup Z^i$
    
    If $\zeta$ belongs to $Z^i$ or if $\zeta$ is an isolated point of 
    $\sigma(\Theta(N)) \cup (Z^{\bb R}_p + i Z^{\bb R}_q)$, then $a\delta_\zeta$
    belongs to $\mc F_{N}$.
\end{example}

\begin{remark}\thlab{taylormehrdimrem}
    Let $h$ be defined on an open subset $D$ of $\bb R^2$ with values in $\bb C$.
    Moreover, assume that for given $m,n\in\bb N$ the function $h$ is 
    $m+n$ times continuously differentiable. Finally, fix $w\in D$.
    
    The well-known Taylors Approximation Theorem from multidimensional calculus then yields
    \[
	h(z) = \sum_{j=0}^{m+n-1} \sum_{\stackrel{k,l\in \bb N_0}{k+l=j}}  
		\frac{1}{k!l!}  
		\frac{\partial^j h}{\partial x^{k} \partial y^{l}}(w)\Re(z-w)^k \Im(z-w)^{l} 
		+ O(|z-w|^{m+n}) 
    \]
    for $z\to w$. Since $|z-w|^{m+n} \leq 2^{m+n}\max(|\Re(z-w)|^{m+n}, |\Im(z-w)|^{m+n}) 
    = O(\max(|\Re(z-w)|^m, |\Im(z-w)|^{n}))$ and since 
    $\Re(z-w)^k \Im(z-w)^{l} = O(\max(|\Re(z-w)|^m, |\Im(z-w)|^{n}))$ for 
    $k \geq m$ or $l\geq n$, we also have
\begin{multline*}
	h(z) = \sum_{k=0}^{m-1} \sum_{l=0}^{n-1} \frac{1}{k!l!}
		\frac{\partial^{k+l} h}{\partial x^{k} \partial y^{l}}(w)\Re(z-w)^k \Im(z-w)^{l} + 
		\\ O(\max(|\Re(z-w)|^m, |\Im(z-w)|^{n})) \,.
\end{multline*}
\end{remark}

\begin{lemma}\thlab{gehzuF}
    Let $f: \dom f \ (\subseteq \bb C^2)\to \bb C$ be a function with the properties
    mentioned in \thref{feinbetef}. Then $f_N$ belongs to $\mc F_{N}$.
\end{lemma}
\begin{proof}
    For fixed $w \in \sigma(\Theta(N)) \cap (Z^{\bb R}_p + i Z^{\bb R}_q)$ and  
    $z\in \sigma(\Theta(N))\setminus (Z^{\bb R}_p + i Z^{\bb R}_q)$ by \thref{taylormehrdimrem}
    the expression
    \[
	  f_N(z) - \sum_{k=0}^{\mf d_p(\Re w)-1}\sum_{l=0}^{\mf d_q(\Im w)-1} 
		f_N(w)_{k,l} \Re(z-w)^k \Im(z-w)^l =
    \]
    \[
	f\circ\tau(z) - 
      \sum_{k=0}^{\mf d_p(\Re w)-1}\sum_{l=0}^{\mf d_q(\Im w)-1} 
		\frac{1}{k!l!} \frac{\partial^{k+l} f\circ \tau}{\partial x^k\partial y^l}(w) \Re(z-w)^k \Im(z-w)^l
    \]
    is a $O(\max(|\Re(z-w)|^{\mf d_p(\Re w)}, |\Im(z-w)|^{\mf d_q(\Im w)}))$ for $z\to w$. 
    Therefore, $f_N\in \mc F_{N}$.
\end{proof}

In order to be able to prove spectral results for our functional calculus, we need that
with $\phi$ also $z\mapsto \phi(z)^{-1}$ belongs to $\mc F_N$ if $\phi$ is bounded away from zero.

\begin{lemma}\thlab{einduF}
    If $\phi \in \mc F_{N}$ is such that $\phi(z)$ is invertible in $\mf C(z)$ (see \thref{z30f3}) for all 
    $z\in \big(\sigma(\Theta(N)) \cup (Z^{\bb R}_p + i Z^{\bb R}_q)\big)\dot\cup Z^i$
    and such that 
    $0$ does not belong to the closure of 
    $\phi\big(\sigma(\Theta(N))\setminus (Z^{\bb R}_p + i Z^{\bb R}_q) \big)$, then
    $\phi^{-1}: z\mapsto \phi(z)^{-1}$ also belongs to $\mc F_{N}$.
\end{lemma}
\begin{proof}
    By the first assumption $\phi^{-1}$ is a well-defined object belonging to $\mc M_{N}$.
    Clearly, with $\phi$ also $z\mapsto \phi(z)^{-1}=\frac{1}{\phi(z)}$ is
    measurable on $\sigma(\Theta(N)) \setminus (Z^{\bb R}_p + i Z^{\bb R}_q)$.
    By the second assumption $z\mapsto \phi(z)^{-1}=\frac{1}{\phi(z)}$ is bounded on this set.
    
    It remains to verify the boundedness of \eqref{fn8qw3b} on a certain neighbourhood of $w$ 
    for each $w \in \sigma(\Theta(N)) \cap (Z^{\bb R}_p + i Z^{\bb R}_q)$ for $\phi^{-1}$. To do so, 
    we calculate for $z\in \sigma(\Theta(N)) \setminus (Z^{\bb R}_p + i Z^{\bb R}_q)$
    \begin{equation}\label{dfbz45pre}
	      \phi^{-1}(z) - 
		\sum_{k=0}^{\mf d_p(\Re w)-1}\sum_{l=0}^{\mf d_q(\Im w)-1} 
		\phi^{-1}(w)_{k,l} \Re(z-w)^k \Im(z-w)^l = 
    \end{equation}
    \begin{equation}\label{dfbz45}
	  \frac{1}{\phi(z)} - 
		\frac{1}{\sum_{k=0}^{\mf d_p(\Re w)-1}\sum_{l=0}^{\mf d_q(\Im w)-1} 
		\phi(w)_{k,l} \Re(z-w)^k \Im(z-w)^l} +
    \end{equation}
\begin{multline}\label{fght33}
	  \frac{1}{\sum_{k=0}^{\mf d_p(\Re w)-1}\sum_{l=0}^{\mf d_q(\Im w)-1} 
		\phi(w)_{k,l} \Re(z-w)^k \Im(z-w)^l} - \\
		\sum_{k=0}^{\mf d_p(\Re w)-1}\sum_{l=0}^{\mf d_q(\Im w)-1} 
		\phi^{-1}(w)_{k,l} \Re(z-w)^k \Im(z-w)^l  
\end{multline}
    The expression in \eqref{dfbz45} can be written as
    \begin{multline*}
	\frac{1}{\phi(z)} \cdot \frac{1}{\sum_{k=0}^{\mf d_p(\Re w)-1}\sum_{l=0}^{\mf d_q(\Im w)-1} 
		\phi(w)_{k,l} \Re(z-w)^k \Im(z-w)^l} \cdot \\
		\left( \phi(z) - \sum_{k=0}^{\mf d_p(\Re w)-1}\sum_{l=0}^{\mf d_q(\Im w)-1} 
		\phi(w)_{k,l} \Re(z-w)^k \Im(z-w)^l \right)
    \end{multline*}
    Here $\frac{1}{\phi(z)}$ is bounded by assumption.
    The assumed invertibility of $\phi(w)$ means $\phi(w)_{0,0}\neq 0$. Hence, 
    \[
	\frac{1}{\sum_{k=0}^{\mf d_p(\Re w)-1}\sum_{l=0}^{\mf d_q(\Im w)-1} 
      \phi(w)_{k,l} \Re(z-w)^k \Im(z-w)^l} = O(1)
    \]
    for $z\to w$.
    From $\phi\in \mc F_N$ we then conclude that \eqref{dfbz45} is a 
    $O(\max(|\Re(z-w)|^{\mf d_p(\Re w)}, |\Im(z-w)|^{\mf d_q(\Im w)}))$ for $z\to w$.

    Because of $\phi(w)\cdot \phi^{-1}=e$ (see \thref{z30f3}), \eqref{fght33} can be rewritten as
    \begin{multline*}
    - \frac{1}{\sum_{k=0}^{\mf d_p(\Re w)-1}\sum_{l=0}^{\mf d_q(\Im w)-1} 
      \phi(w)_{k,l} \Re(z-w)^k \Im(z-w)^l}\cdot \\
      \Big(
      \sum_{k=0}^{\mf d_p(\Re w)-1}\sum_{l=0}^{\mf d_q(\Im w)-1} 
      \Re(z-w)^k \Im(z-w)^l \cdot  \sum_{c=0}^k \sum_{d=0}^l \phi(w)_{c,d} \cdot \phi^{-1}(w)_{k-c,l-d}
      \\ + O(\max(|\Re(z-w)|^{\mf d_p(\Re w)}, |\Im(z-w)|^{\mf d_q(\Im w)})) - 1\Big) = 
    \end{multline*}
    \begin{multline*}
	O(1) \cdot O(\max(|\Re(z-w)|^{\mf d_p(\Re w)}, |\Im(z-w)|^{\mf d_q(\Im w)})) = \\
	O(\max(|\Re(z-w)|^{\mf d_p(\Re w)}, |\Im(z-w)|^{\mf d_q(\Im w)})) 
    \end{multline*}
    for $z\to w$. Altogether \eqref{dfbz45pre} is a 
    $O(\max(|\Re(z-w)|^{\mf d_p(\Re w)}, |\Im(z-w)|^{\mf d_q(\Im w)}))$. Therefore, 
    $\phi^{-1}\in \mc F_N$.
\end{proof}

\section{Functional Calculus}

In this section we employ the same assumptions and notation as in the previous one.

\begin{lemma}\thlab{aufspaltb}
     For any $\phi\in \mc F_{N}$ there exists a polynomial $s\in \mathbb C[z,w]$
     and a function $g$ on $\sigma(\Theta(N))$ with values in 
     $\bb C$ on $\sigma(\Theta(N))\setminus (Z^{\bb R}_p + i Z^{\bb R}_q)$
     and values in $\bb C^2$ on $\sigma(\Theta(N))\cap (Z^{\bb R}_p + i Z^{\bb R}_q)$ such that
     $\phi - s_N\in \mc R$, such that
     $g$ is bounded and measurable on $\sigma(\Theta(N))\setminus (Z^{\bb R}_p + i Z^{\bb R}_q)$, and
     such that
     \begin{equation}\label{dbew66290}
	\phi(z) = s_N(z) + 
	(p_N + q_N) (z)\cdot g(z), \ z \in \sigma(\Theta(N)) \,,
     \end{equation}
     where the multiplication here has to be understood in the sense of \thref{durchdiv}.
\end{lemma}
\begin{proof}
    According to \thref{existbe} there exists an $s\in \mathbb C[z,w]$ such that
    $\phi - s_N \in \mc R$, and by \thref{durchdiv} we then find a function $g$ 
    such that \eqref{dbew66290} holds true.     
    The measurability of 
    \[
	g(z)=\frac{\phi(z) - s(\Re z,\Im z)}{p(\Re z)+q(\Im z)} \ \text{ on } \ 
	      \sigma(\Theta(N))\setminus (Z^{\bb R}_p + i Z^{\bb R}_q)
    \]
    follows from the assumption $\phi\in\mc F_{N}$; see \thref{FdefklM}.
    
    In order to show $g$'s boundedness, first recall from \thref{speknorm} that
    \[
	\max(|p(\Re z)|,|q(\Im z)|) \leq \max(\|R_1R_1^*\|,\|R_2R_2^*\|) \, |p(\Re z) + q(\Im z)| 
    \]
    for $z\in \sigma(\Theta(N))$.
    Hence, 
    \begin{multline*}
%    \[
      \frac{\max(|p(\Re z)|,|q(\Im z)|)}{|p(\Re z) + q(\Im z)|} \leq \max(\|R_1R_1^*\|,\|R_2R_2^*\|), \\
      z\in \sigma(\Theta(N))\setminus (Z^{\bb R}_p + i Z^{\bb R}_q) \,.
    \end{multline*}
%    \]
    As $\phi\in \mc F_{N}$ we find for each $w\in \sigma(\Theta(N))\cap (Z^{\bb R}_p + i Z^{\bb R}_q)$
    an open neighbourhood $U(w)$ of $w$ such that \eqref{fn8qw3b} is bounded for $z \in U(w)\setminus \{w\}$.
    Clearly, we can make the neighbourhoods $U(w)$ smaller so that they are pairwise disjoint.
    Since for $w\in \sigma(\Theta(N))\cap (Z^{\bb R}_p + i Z^{\bb R}_q)$ the real number $\Re w$ ($\Im w$) 
    is a zero of $p(\Re z)$ ($q(\Im z)$) with multiplicity $\mf d_p(\Re w)$ ($\mf d_q(\Im w)$), we have
    \[
	  c |\Re(z-w)|^{\mf d_p(\Re w)} \leq |p(\Re z)|, \ \
	  d |\Im(z-w)|^{\mf d_q(\Im w)} \leq |q(\Im z)| 
    \]
    for $z \in U(w)$ with constants $c, d > 0$. Hence,
    \[
    \frac{\max(|\Re(z-w)|^{\mf d_p(\Re w)},|\Im(z-w)|^{\mf d_q(\Im w)})}{\max(|p(\Re z)|,|q(\Im z)|)} \leq C_w
    \]
    on $\sigma(\Theta(N))\cap U(w)\setminus \{w\}$ for some $C_w > 0$. 
    By what was said in \thref{taylormehrdimrem} and we also have
%     \begin{multline*}
% 	s(\Re z,\Im z) = \\
% 		  \sum_{k=0}^{\mf d_p(\Re w)-1}\sum_{l=0}^{\mf d_q(\Im w)-1} 
% 		\frac{1}{k!l!} \, \frac{\partial^{k+l}}{\partial x^k\partial y^l} 
% 		      s(\Re w,\Im w) \Re(z-w)^k \Im(z-w)^l + \\
% 		      O\big(\max(|\Re(z-w)|^{\mf d_p(\Re w)},|\Im(z-w)|^{\mf d_q(\Im w)})\big) =
%     \end{multline*}
    \begin{multline*}
		 s(\Re z,\Im z) = \sum_{k=0}^{\mf d_p(\Re w)-1}\sum_{l=0}^{\mf d_q(\Im w)-1} 
		\phi(w)_{k,l} \Re(z-w)^k \Im(z-w)^l + \\
		      O\big(\max(|\Re(z-w)|^{\mf d_p(\Re w)},|\Im(z-w)|^{\mf d_q(\Im w)})\big)
		\,,
    \end{multline*}
    because $\phi - s_N \in \mc R$ implies $\phi(w)_{k,l} = 
    \frac{1}{k!l!} \frac{\partial^{k+l} s}{\partial x^k\partial y^l}(\Re w,\Im w)$. 
    Using the boundedness of \eqref{fn8qw3b} we altogether obtain the boundedness of
    \begin{equation}\label{nurx549t6}
	g(z) = \frac{\phi(z) - s(\Re z,\Im z)}{p(\Re z)+q(\Im z)} = 
    \end{equation}
\begin{multline*}
	\frac{\max(|p(\Re z)|,|q(\Im z)|)}{p(\Re z) + q(\Im z)} \cdot \\
	\frac{\max(|\Re(z-w)|^{\mf d_p(\Re w)},|\Im(z-w)|^{\mf d_q(\Im w)})}{\max(|p(\Re z)|,|q(\Im z)|)} \cdot \\
	\frac{\phi(z) - s(\Re z,\Im z)}{\max(|\Re(z-w)|^{\mf d_p(\Re w)},|\Im(z-w)|^{\mf d_q(\Im w)})} 
\end{multline*}
    for $z\in \sigma(\Theta(N))\cap U(w)\setminus \{w\}$. 
    Since by \thref{speknorm} the function
    $\frac{1}{p(\Re z)+q(\Im z)}$ is continuous, and hence
    bounded on $\sigma(\Theta(N)) \setminus \bigcup_{w\in \sigma(\Theta(N))\cap (Z^{\bb R}_p + i Z^{\bb R}_q)} U(w)$, 
    we see that
    \eqref{nurx549t6} is even bounded for $z\in \sigma(\Theta(N)) \setminus (Z^{\bb R}_p + i Z^{\bb R}_q)$.
\end{proof}

\begin{definition}\thlab{funcaldef}
%     Let $N = A + i B$ ($A=A^*,B=B^*$) be a definitizable normal operator on a Krein space $\mc K$, and let
%     $p\in \bb R[z]$ ($q\in \bb R[z]$) be such that $[p(A)x,x]\geq 0, \ x \in \mc K$ ($[q(A)x,x]\geq 0, \ x \in \mc K$).
    For any $\phi\in \mc F_{N}$ we define 
\[
    \phi(N):= s(A,B) + \Xi\left( \int^{R_1,R_2}_{\sigma(\Theta(N))} g \, dE \right) \,,
\]
     where $s\in \bb C[z,w]$ and $g$ is a function on $\sigma(\Theta(N))$ with the properties mentioned in 
     \thref{aufspaltb}, and where
\begin{multline*}
    \int^{R_1,R_2}_{\sigma(\Theta(N))} g \, dE := 
    \int_{\sigma(\Theta(N))\setminus (Z^{\bb R}_p + i Z^{\bb R}_q)} g \, dE 
    \ +  \\ \hspace*{5mm}
     \sum_{w\in \sigma(\Theta(N)) \cap (Z^{\bb R}_p + i Z^{\bb R}_q)} 
	  \big(g(w)_1 R_1R_1^* E\{w\} + g(w)_2 R_2R_2^* E\{w\}\big)  \,. 
\end{multline*}
\end{definition}

First we shall show that $\phi(N)$ is well defined.

\begin{theorem}\thlab{welldef}
    Let $\phi \in \mc F_{N}$, $s, \tilde s\in \mathbb C[z,w]$ and 
    functions $g,\tilde g$ on $\sigma(\Theta(N))$ be given, such that
    the assertion of \thref{aufspaltb} holds true for $s, g$ as well as for
    $\tilde s, \tilde g$. Then
\[
    s(A,B) + \Xi\left(\int^{R_1,R_2}_{\sigma(\Theta(N))} g \, dE \right) = 
    \tilde s(A,B) + \Xi\left(\int^{R_1,R_2}_{\sigma(\Theta(N))} \tilde g \, dE \right) \,. 
\]
\end{theorem}
\begin{proof}
  By assumption we have $\phi - s_N, \phi - \tilde s_N \in \mc R$. Subtracting
  these functions yields $\tilde s_N - s_N \in \mc R$. Using the notation of \thref{einbett2pre} this gives
  $\varpi(\tilde s - s)_{\xi,\eta} = 0$ for $(\xi,\eta) \in p^{-1}\{0\}\times q^{-1}\{0\}$.
  According to \thref{einbett2pre} we then get 
    \begin{equation}\label{opzerlsssch}
  \tilde s(z,w) - s(z,w) = p(z) u(z,w) + q(w) v(z,w)     
    \end{equation}
  for some $u(z,w), v(z,w) \in \bb C[z,w]$. 
  
  By \thref{Xidefeig} in \cite{KaPr2014} we have
  \[
      \Xi_1\big(u(\Theta_1(A),\Theta_1(B))\big) = \Xi_1\big(\Theta_1(u(A,B))\big) = p(A) u(A,B) \,,
  \]
  \[
      \Xi_2\big(v(\Theta_2(A),\Theta_2(B))\big) = \Xi_2\big(\Theta_2(v(A,B))\big) = q(B) v(A,B) \,,
  \]
  where $\Xi_j, \ j=1,2$, are as defined in \eqref{xijdef}.
  Since $u(\Theta_1(A),\Theta_1(B)) = \int u(\Re z,\Im z) \, dE_1(z)$, we get from \eqref{zuef2}
\[
      \Xi_1\big(u(\Theta_1(A),\Theta_1(B))\big) = 
      \Xi\big(R_1R_1^* \int u(\Re z,\Im z) \, dE(z)\big) \,.
\]
  Similarly, $\Xi_2\big(v(\Theta_2(A),\Theta_2(B))\big) = \Xi\big(R_2R_2^* \int v(\Re z,\Im z) \, dE(z)\big)$.
  Therefore, employing \thref{korvda} we get
\begin{multline}\label{tismis}
      \tilde s(A,B) - s(A,B) = p(A) u(A,B) + q(B) v(A,B) =
      \\ \Xi\big(R_1R_1^* \int u(\Re z,\Im z) \, dE(z) + R_2R_2^* \int v(\Re z,\Im z) \, dE(z)\big) =
\end{multline}
\begin{multline*}
      \Xi\left( \int_{\sigma(\Theta(N))\setminus (Z^{\bb R}_p + i Z^{\bb R}_q)} \hspace*{-5mm}
	    \frac{p(\Re z)u(\Re z,\Im z) + q(\Im z)v(\Re z,\Im z)}{p(\Re z)+q(\Im z)} \, dE(z) + \right. \\
		   \left. \sum_{w\in \sigma(\Theta(N)) \cap (Z^{\bb R}_p + i Z^{\bb R}_q)} \hspace*{-5mm}
			\big( u(\Re w,\Im w) R_1R_1^* E\{w\} + v(\Re w,\Im w) R_2R_2^* E\{w\}\big)
		    \right)\,.
\end{multline*}
  On the other hand, since \eqref{dbew66290} holds true for $s,g$ and $\tilde s,\tilde g$, we have
  \begin{equation}\label{kkwae}
      (\tilde s_N -s_N)(z) = (p_N + q_N) (z)\cdot (g(z)-\tilde g(z)), \ z \in z \in \sigma(\Theta(N)) \,.
  \end{equation}
  For $z\in \sigma(\Theta(N)) \setminus (Z^{\bb R}_p + i Z^{\bb R}_q)$ by \eqref{opzerlsssch} this means
\begin{multline*}
      p(\Re z) u(\Re z,\Im z) + q(\Im z) v(\Re z,\Im z) 
      = \\ \tilde s(\Re z,\Im z) - s(\Re z,\Im z) = (p(\Re z) + q(\Im z)) \cdot (g(z)-\tilde g(z)) 
\end{multline*}
  and, in turn, 
  \[
      g(z)-\tilde g(z) = \frac{p(\Re z) u(\Re z,\Im z) + q(\Im z) v(\Re z,\Im z)}{p(\Re z) + q(\Im z)} \,.
  \]
  Considering for $z\in \sigma(\Theta(N)) \cap (Z^{\bb R}_p + i Z^{\bb R}_q)$ the
  entries of \eqref{kkwae} with indices $(\mf d_p(\Re z),0)$ and $(0,\mf d_q(\Im z))$ 
  together with \eqref{opzerlsssch} we get
\begin{multline*}
      \frac{1}{\mf d_p(\Re z)!} \, p^{(\mf d_p(\Re z))}(\Re z) \, u(\Re z,\Im z) = \\
      \frac{1}{\mf d_p(\Re z)!} \, \frac{\partial^{\mf d_p(\Re z)}}{\partial x^{\mf d_p(\Re z)}}      
      (\tilde s(\Re z,\Im z) - s(\Re z,\Im z)) = \\ \frac{1}{\mf d_p(\Re z)!} \, p^{(\mf d_p(\Re z))}(\Re z) \, 
      (g(z)_1-\tilde g(z)_1)
\end{multline*}
  and
\begin{multline*}
      \frac{1}{\mf d_q(\Im z)!} \, q^{(\mf d_q(\Im z))}(\Im z) \, v(\Re z,\Im z) = \\
      \frac{1}{\mf d_q(\Im z)!} \, \frac{\partial^{\mf d_q(\Im z)}}{\partial y^{\mf d_q(\Im z)}}      
      (\tilde s(\Re z,\Im z) - s(\Re z,\Im z)) = \\ \frac{1}{\mf d_q(\Im z)!} \, q^{(\mf d_q(\Im z))}(\Im z) 
	  \, (g(z)_2-\tilde g(z)_2)
\end{multline*}
  where we employed the product rule and the fact that 
  $p^{(k)}(\Re z) = 0 = q^{(l)}(\Im z)$ for $0\leq k < \mf d_p(\Re z), \, 0\leq l < \mf d_q(\Im z)$. 
  Since $p^{(\mf d_p(\Re z))}(\Re z)$ and $q^{(\mf d_q(\Im z))}(\Im z)$ do not vanish for 
  $z\in \sigma(\Theta(N)) \cap (Z^{\bb R}_p + i Z^{\bb R}_q)$, we get 
  $u(\Re z,\Im z) = g(z)_1-\tilde g(z)_1$ and $v(\Re z,\Im z) = g(z)_2-\tilde g(z)_2$.
  Therefore, we can write \eqref{tismis} as
  \[
  \tilde s(A,B) - s(A,B) = \Xi\left( \int^{R_1,R_2}_{\sigma(\Theta(N))} \big(g-\tilde g\big) \, dE \right)\,,
  \]
  showing the asserted equality.
\end{proof}

\begin{theorem}\thlab{mimalvertr}
    The mapping $\phi\mapsto \phi(N)$ constitutes a $*$-homomorphism from 
    $\mc F_{N}$ into $\{N,N^*\}'' \ (\subseteq B(\mc K))$ with $s_N(N) = s(A,B)$ for all $s\in \bb C[z,w]$.
\end{theorem}
\begin{proof}
    $s_N(N) = s(A,B)$ for all $s\in \bb C[z,w]$ follows from \thref{welldef} because 
    we have $s_N = s_N + (p_N + q_N)(z)\cdot 0, \, z \in \sigma(\Theta(N))$.

    Assume that for $\phi,\psi \in\mc F_{N}$ we have $s,r\in \bb C[z,w]$ and functions $g,h$ on $\sigma(\Theta(N))$
    such that $\phi - s_N, \psi-r_N\in \mc R$, 
    such that $g$ and $h$ are bounded and measurable on $\sigma(\Theta(N))\setminus (Z^{\bb R}_p + i Z^{\bb R}_q)$,
    and such that \eqref{dbew66290} as well as
    \[
	\psi(z) = r_N(z) + 
	(p_N + q_N)(z)\cdot h(z), \ z \in \sigma(\Theta(N)) \,,
    \]
    hold true; see \thref{aufspaltb}. Then for $\lambda,\mu\in\bb C$ we get from \thref{bweuh30}
    \[
	(\lambda \phi + \mu \psi)(z) = (\lambda s + \mu r)_N(z) + 
	(p_N + q_N)(z)\cdot (\lambda g(z) + \mu h(z)), \ z \in \sigma(\Theta(N)) \,, 
    \]
    where 
    $\lambda \phi + \mu \psi - (\lambda s + \mu r)_N = 
    \lambda(\phi-s_N) + \mu(\psi-r_N) \in \mc R$, and where
    $\lambda g + \mu h$ is bounded and measurable on $\sigma(\Theta(N))\setminus (Z^{\bb R}_p + i Z^{\bb R}_q)$.
    Since the definition of $\phi(N)$ in \thref{funcaldef} depends linearly on $s$ and $g$,
    we conclude from \thref{welldef} that 
    \[
      (\lambda \phi + \mu \psi)(N)=\lambda \phi(N) + \mu \psi(N) \,.
    \]
    Similarly, we get $\phi^\#(z) = (s^\#)_N(z) + 
    (p_N + q_N)(z)\cdot \bar g(z), \ z \in \sigma(\Theta(N))$; see \thref{bweuh30}.
    Thereby $\phi^\# - (s^\#)_N = (\phi - s_N)^\#\in \mc R$ holds true 
    due to the fact that $\mf d_p(\xi)=\mf d_p(\bar \xi)$
    and $\mf d_q(\eta)=\mf d_q(\bar \eta)$ for all $(\xi,\eta)\in Z^i$.
    Since $\bar g$ is bounded and measurable on $\sigma(\Theta(N))\setminus (Z^{\bb R}_p + i Z^{\bb R}_q)$,
    and since 
    \[
    \phi(N)^*= s^\#(A,B) + \Xi\left(\int^{R_1,R_2}_{\sigma(\Theta(N))} \bar g \, dE \right) \,, 
    \]
    we again obtain from \thref{welldef} that $\phi^\#(N)=\phi(N)^*$.
    
    Concerning the compatibility with $\cdot$, first note that by \thref{bweuh30}
    \[
	\phi(z) \cdot \psi(z) = (s\cdot r)_N(z) + 
	(p_N + q_N)(z)\cdot \omega(z), \ z \in \sigma(\Theta(N)) \,. 
    \]
    Here we have $\omega(z) = s(z) h(z) + r(z) g(z) + g(z)h(z)( p(\Re z) + q(\Im z))$ for 
    $z\in \sigma(\Theta(N))\setminus (Z^{\bb R}_p + i Z^{\bb R}_q)$
    and $\omega(z)_j = s(z) g(z)_j + r(z) h(z)_j, \ j=1,2$ for 
    $z\in \sigma(\Theta(N))\cap (Z^{\bb R}_p + i Z^{\bb R}_q)$ because
    $a,b \in \ker \pi$ implies $a\cdot b=0$ and, in turn, 
    $(p_N + q_N)(z) \cdot (p_N + q_N)(z) = 0$ for 
    $z\in \sigma(\Theta(N))\cap (Z^{\bb R}_p + i Z^{\bb R}_q)$.
    
    On the other hand, by \thref{Xidefeig} in \cite{KaPr2014}
    we have $\Xi(D)C = \Xi(D\Theta(C))$, $C\Xi(D) = \Xi(\Theta(C)D)$, and
    $\Xi(D_1)\Xi(D_2) = \Xi(D_1D_2 T^*T)$, where $T^*T=p(A) + q(B)$.
    Hence,
    \[
	\phi(N) \ \psi(N) =  
     \]\[
	s(A,B) \, r(A,B) + \Xi\left(\int^{R_1,R_2}_{\sigma(\Theta(N))} g \, dE \right) r(A,B) +
    s(A,B) \, \Xi\left(\int^{R_1,R_2}_{\sigma(\Theta(N))} h \, dE \right) +
     \]\[
    \Xi\left(\int^{R_1,R_2}_{\sigma(\Theta(N))} g \, dE \right) \,
     \Xi\left(\int^{R_1,R_2}_{\sigma(\Theta(N))} h \, dE \right) 
     =
     \]
\begin{multline*}
     (s \cdot r)(A,B) + 
     \Xi\left(
     \int^{R_1,R_2}_{\sigma(\Theta(N))} (g\cdot r + h \cdot s) \, dE + \right.
          \\
      \left. \int_{\sigma(\Theta(N))\setminus (Z^{\bb R}_p + i Z^{\bb R}_q)} 
			  \big(p(\Re(.)) + q(\Im(.))\big) \cdot h\cdot g \, dE\right) =
\end{multline*}
    \[
	(s \cdot r)(A,B) + \Xi\left(\int^{R_1,R_2}_{\sigma(\Theta(N))} \omega \, dE\right) \,.
    \]
    Here $\omega$ is bounded and measurable on $\sigma(\Theta(N))\setminus (Z^{\bb R}_p + i Z^{\bb R}_q)$
    and, using the fact that $\mc R$ is an ideal, 
    \[
	\phi \cdot \psi - (s\cdot r)_N = (\phi-s_N)\cdot \psi + (\psi-r_N)\cdot s_N \in \mc R \,. 
    \]    
    Hence, we again obtain from 
    \thref{welldef} that $\phi(N)\cdot \psi(N)=\big(\phi\cdot\psi\big)(N)$.
    
    Finally, we shall show that $\phi(N)\in \{N,N^*\}''$.
    Clearly, $s(A,B) \in \{A,B\}''=\{N,N^*\}''$. If $C\in \{A,B\}' \subseteq \big(p(A)+q(B)\big)' = (TT^*)'$, 
    then $\Theta(C) \in \{\Theta(A),\Theta(B)\}'$ because $\Theta$ is a homomorphism. By the 
    spectral theorem for normal operators $\Theta(C)$ commutes with 
    \[
	  D:=\int^{R_1,R_2}_{\sigma(\Theta(N))} g \, dE \,.
    \]
    According to \thref{Xidefeig} in \cite{KaPr2014} we then get 
    \[
	\Xi(D) C = \Xi(D\Theta(C)) = \Xi(\Theta(C) D) = C \Xi(D) \,.
    \]
    Hence, $\Xi(D) \in \{A,B\}''=\{N,N^*\}''$, and altogether $\phi(N) \in \{A,B\}''=\{N,N^*\}''$. 
\end{proof}

\begin{remark}\thlab{rieszproj}
For $\zeta \in Z^i$ or for an isolated $\zeta \in \sigma(\Theta(N)) \cup (Z^{\bb R}_p + iZ^{\bb R}_q)$ 
we saw in \thref{fedela} that $a\delta_\zeta \in \mc F_{N}$.
If $a$ is the unite $e \in \mf C(\zeta)$ (see \thref{z30f3}), then 
$(e\delta_\zeta)\cdot (e\delta_\zeta) = (e\delta_\zeta)$ together with \thref{mimalvertr} shows that
$(e\delta_\zeta)(N)$ is a projection. It is a kind of Riesz projection corresponding to $\zeta$. 

We set $\xi:=\Re \zeta, \ \eta:=\Im \zeta$ if $\zeta \in \sigma(\Theta(N)) \cup (Z^{\bb R}_p + iZ^{\bb R}_q)$ and
$(\xi,\eta):=\zeta$ if $\zeta \in Z^i$.
For $\lambda \in \bb C\setminus \{\xi+i\eta\}$ and for $s(z,w):=z+iw-\lambda$ we then have
$s_N \cdot (e\delta_\zeta) = \big(s_N(\zeta)\big)\delta_\zeta$, where
the entry $s(\xi,\eta)$ of $s_N(\zeta)$ with index $(0,0)$ does not vanish.
By \thref{z30f3} it therefore has a multiplicative inverse $b\in \mf C(\zeta)$. We then obtain
\[
    s_N \cdot (e\delta_\zeta) \cdot (b\delta_\zeta) = e\delta_\zeta \,.
\]
From $s_N(N) = N - \lambda$ we then get that $N\vert_{\ran (e\delta_\zeta)(N)} - \lambda$
has $(b\delta_\zeta)(N)\vert_{\ran (e\delta_\zeta)(N)}$ as its inverse operator. 
Thus, $\sigma(N\vert_{\ran (e\delta_\zeta)(N)}) \subseteq \{\xi+i\eta\}$.
\end{remark}

\begin{lemma}\thlab{deth56}
    If for $\phi\in \mc F_{N}$ we have $\phi(z)= 0$ 
    for all 
    \[
      z \in \big(\sigma(\Theta(N)) \cup ((Z^{\bb R}_p +i Z^{\bb R}_q) \cap \sigma(N))\big) 
	  \dot\cup \{(\alpha,\beta)\in Z^i: \alpha+i\beta, \bar \alpha+i\bar \beta \in \sigma(N)\} \,, 
    \]
    then $\phi(N)=0$.
\end{lemma}
\begin{proof}
    Since any $\zeta \in (Z^{\bb R}_p +i Z^{\bb R}_q) \setminus \sigma(N)$ is isolated
    in $\sigma(\Theta(N)) \cup (Z^{\bb R}_p + iZ^{\bb R}_q)$, we saw in \thref{rieszproj} that
    for 
    \[
	\zeta \in \underbrace{\big((Z^{\bb R}_p +i Z^{\bb R}_q) \setminus \sigma(N)\big)}_{=:Z_1} \dot\cup
	\underbrace{\{(\alpha,\beta)\in Z^i: \alpha+i\beta \in \rho(N)\}}_{=:Z_2}
    \]
    the expression 
    $(e\delta_\zeta)(N)$ is a bounded projection commuting with $N$. Hence,
    $(e\delta_\zeta)(N)$ also commutes $(N - (\xi+i\eta))^{-1}$, where
    $\xi:=\Re \zeta, \ \eta:=\Im \zeta$ if $\zeta \in Z_1$ and
    $(\xi,\eta):=\zeta$ if $\zeta \in Z_2$.
    
    Consequently, $N\vert_{\ran (e\delta_\zeta)(N)}- (\xi+i\eta)$ is invertible on $\ran (e\delta_\zeta)(N)$,
    i.e.\ $\xi+i\eta \not\in \sigma(N\vert_{\ran (e\delta_\zeta)(N)})$. By
    \thref{rieszproj} we have $\sigma(N\vert_{\ran (e\delta_\zeta)(N)}) \subseteq \{\xi+i\eta\}$.
    Hence, $\sigma(N\vert_{\ran (e\delta_\zeta)(N)}) = \emptyset$, which is impossible for
    $\ran (e\delta_\zeta)(N) \neq \{0\}$. Thus, $(e\delta_\zeta)(N) = 0$. 
    
    For $(\xi,\eta)\in Z_3:=\{(\alpha,\beta)\in Z^i: \bar\alpha+i\bar\beta \in \rho(N)\}$
    we get $(\bar \xi,\bar \eta)\in Z_2$. Hence,
    \[
	0=(e\delta_{(\bar\xi,\bar\eta)})(N)^*=(e\delta_{(\xi,\eta)})(N) \,.
    \]
    By our assumption $\phi$ is supported on $Z_1\cup Z_2 \cup Z_3$. Hence, 
    \[
	\phi(N) = (\hspace{-2mm}\sum_{\zeta \in Z_1\cup Z_2 \cup Z_3} \hspace{-2mm} 
	\phi(\zeta)\delta_\zeta \hspace{+2mm})(N) = 
	\sum_{\zeta \in Z_1\cup Z_2 \cup Z_3} \phi(\zeta) (e\delta_\zeta)(N) = 0 \,. 
    \]
 \end{proof}

\begin{remark}\thlab{fhwr34}
As a consequence of \thref{deth56} for $\phi \in \mc F_{N}$ the operator $\phi(N)$ only depends on
$\phi$'s values on
  \begin{multline*}%\label{eigmeng}
  \sigma_N:=
\big(\sigma(\Theta(N)) \cup ((Z^{\bb R}_p +i Z^{\bb R}_q) \cap \sigma(N))\big) \dot\cup \\ \{(\alpha,\beta)\in Z^i: 
      \alpha+i\beta, \bar \alpha+i\bar \beta \in \sigma(N)\} 
  \end{multline*}
Thus, we can re-define the function class $\mc F_{N}$ for our functional calculus
so that the elements $\phi$ of $\mc F_{N}$ are functions on this set
with $\phi(z)\in \mf C(z)$
such that $z\mapsto \phi(z)$ is measurable and bounded on $\sigma(\Theta(N) \setminus (Z^{\bb R}_p +i Z^{\bb R}_q)$
and such that \eqref{fn8qw3b} is bounded locally at $w$ 
for all $w\in \sigma(\Theta(N) \cap (Z^{\bb R}_p +i Z^{\bb R}_q)$.
\end{remark}

\begin{lemma}\thlab{wannboundinv}
    If $\phi \in \mc F_{N}$ 
    is such that $\phi(z)$ is invertible in $\mf C(z)$ (see \thref{z30f3}) for all 
    $z\in \sigma_N$
    and such that 
    $0$ does not belong to the closure of 
    $\phi\big(\sigma(\Theta(N))\setminus (Z^{\bb R}_p + i Z^{\bb R}_q) \big)$, then
    $\phi(N)$ is a boundedly invertible operator on $\mc K$.
\end{lemma}
\begin{proof}
    We think of $\phi$ as a function on $\big(\sigma(\Theta(N)) \cup (Z^{\bb R}_p + i Z^{\bb R}_q)\big)\dot\cup Z^i$
    by setting $\phi(z)=e$ (see \thref{z30f3}) for all $z$ not belonging to $\sigma_N$.
    Then all assumptions of \thref{einduF} are satisfied. Hence $\phi^{-1} \in \mc F_{N}$, and we conclude from
    \thref{mimalvertr} and \thref{fuinm0pre} that
    \[
	\phi^{-1}(N) \phi(N) = \phi(N) \phi^{-1}(N) = (\phi\cdot \phi^{-1})(N) = \mathds{1}_N(N) = I \,. 
    \]
\end{proof}

\begin{corollary}\thlab{sigmaN}
    If $N$ is a definitizable normal operator on the Krein space $\mc K$, then $\sigma(N)$
    equals to
  \begin{multline}\label{specform}
    \sigma(\Theta(N)) \cup
	    ((Z^{\bb R}_p +i Z^{\bb R}_q) \cap \sigma(N))
    \cup \\ \{\alpha + i\beta :  (\alpha,\beta)\in Z^i, 
      \alpha+i\beta, \bar \alpha+i\bar \beta \in \sigma(N)\} 
  \end{multline}
\end{corollary}
\begin{proof}
    Since $\Theta$ is a homomorphism, we have $\sigma(\Theta(N))\subseteq \sigma(N)$.
    Hence, \eqref{specform} is contained in $\sigma(N)$.
    
    For the converse, consider the polynomial $s(z,w) = z+iw - \lambda$ for a $\lambda$
    not belonging to \eqref{specform}. We conclude that for any $z\in \sigma_N$ the first entry 
    $(s_N(z))_{0,0}$ of $s_N(z)\in \mf C(z)$ does not vanish, i.e.\ is invertible in $\mf C(z)$. 
    $(s_N(\sigma(\Theta(N))))_{0,0} = \sigma(\Theta(N)) - \lambda$ being compact,
    $0$ does not belong to the closure of 
    $s_N\big(\sigma(\Theta(N))\setminus (Z^{\bb R}_p + i Z^{\bb R}_q) \big)$.
    
    Applying \thref{wannboundinv} we see that $s_N(N)=(N-\lambda)$ is invertible.
\end{proof}

\begin{corollary}
For $\phi\in \mc F_{N}$ we have 
\[
	\sigma(\phi(N)) \subseteq \overline{\phi(\sigma_N)_{0,0}} \,.
\]
\end{corollary}
\begin{proof}
For $\lambda\notin \overline{\phi(\sigma_N)_{0,0}}$ 
and any $z\in \sigma_N$ we have 
$(\phi(z)- \lambda \mathds{1}_N(z))_{0,0} = \phi(z)_{0,0} - \lambda \neq 0$.
Hence $\phi(z)- \lambda \mathds{1}_N(z)$ is invertible in $\mf C(z)$.

Moreover, $0$ does not belong to the closure of 
$\phi\big(\sigma(\Theta(N))\setminus (Z^{\bb R}_p + i Z^{\bb R}_q) \big) -\lambda =
(\phi- \lambda \mathds{1}_N)\big(\sigma(\Theta(N))\setminus (Z^{\bb R}_p + i Z^{\bb R}_q) \big)_{0,0}$.
Therefore, we can apply \thref{wannboundinv} to $\phi- \lambda \mathds{1}_N$, and get
$\lambda \in \rho(\phi(N))$.
\end{proof}

\begin{remark}\thlab{prf55o1}
For
any characteristic function $\mathds{1}_\Delta$ of a Borel subset $\Delta\subseteq \bb C$ such that
$(Z^{\bb R}_p +i Z^{\bb R}_q) \cap \sigma(N) \cap \partial_{\bb C} \Delta = \emptyset$ the function 
$(\mathds{1}_{\tau(\Delta)})_{N}$ belongs to $\mc F_N$; see \thref{feinbetef} and \thref{gehzuF}.
Since this function is idempotent and satisfies $(\mathds{1}_{\tau(\Delta)})_{N}^\# = (\mathds{1}_{\tau(\Delta)})_{N}$,
$(\mathds{1}_{\tau(\Delta)})_{N}(N)$ is a bounded and self-adjoint projection on the Krein space $\mc K$.
These projections constitute the family of spectral projections for $N$.
\end{remark}

\end{document}